\documentclass[11pt, draft]{amsart}
\usepackage{amssymb, amstext, amscd, amsmath, amssymb}
\usepackage{mathtools, xypic, paralist, color, dsfont, rotating}
\usepackage{verbatim}
\usepackage{enumerate,enumitem}
\usepackage{float}
\usepackage{setspace}
\usepackage{pifont}

\usepackage{tikz}
\usetikzlibrary{fit,positioning,arrows,automata,calc}
\tikzset{
  main/.style={circle, minimum size = 30pt, thick, draw = black!80, node distance = 10mm},
  connect/.style={-latex, thick},
  box/.style={rectangle, draw = white!100}
}
\usepackage{tikz-cd}

\usepackage{tikz}

\floatstyle{boxed} 
\restylefloat{figure}

\numberwithin{equation}{section}
\usepackage[a4paper]{geometry}
\usepackage{quoting}
\quotingsetup{vskip=.1in}
\quotingsetup{leftmargin=.6in} 
\quotingsetup{rightmargin=.6in}

\let\OLDthebibliography\thebibliography
\renewcommand\thebibliography[1]{
  \OLDthebibliography{#1}
  \setlength{\parskip}{0pt}
  \setlength{\itemsep}{2pt plus 0.5ex}
}

\makeatletter
\def\@cite#1#2{{\m@th\upshape\bfseries%
[{#1\if@tempswa{\m@th\upshape\mdseries, #2}\fi}]}}
\makeatother
\theoremstyle{plain}
\newtheorem{theorem}{Theorem}[section]
\newtheorem{corollary}[theorem]{Corollary}
\newtheorem{proposition}[theorem]{Proposition}
\newtheorem{lemma}[theorem]{Lemma}
\theoremstyle{definition}
\newtheorem{definition}[theorem]{Definition}
\newtheorem{example}[theorem]{Example}

\newtheorem{remark}[theorem]{Remark}

\newtheorem*{open}{Open Access Statement}
\theoremstyle{remark}

\mathtoolsset{centercolon}
  \newcommand{\A}{{\mathcal{A}}}
  \newcommand{\B}{{\mathcal{B}}}
  \newcommand{\C}{{\mathcal{C}}}

  \newcommand{\F}{{\mathcal{F}}}
  \newcommand{\G}{{\mathcal{G}}}

  \newcommand{\K}{{\mathcal{K}}}
\renewcommand{\L}{{\mathcal{L}}}

\renewcommand{\O}{{\mathcal{O}}}

\renewcommand{\S}{{\mathcal{S}}}
  \newcommand{\T}{{\mathcal{T}}}

\newcommand{\eps}{\varepsilon}
\def\al{\alpha}
\def\be{\beta}

\def\de{\delta}

\def\io{\iota}
\def\ka{\kappa}
\def\la{\lambda}

\def\si{\sigma}

\newcommand\vphi{\varphi}

\newcommand{\bC}{\mathbb{C}}

\newcommand{\bN}{\mathbb{N}}
\newcommand{\bT}{\mathbb{T}}
\newcommand{\bZ}{\mathbb{Z}}

\newcommand{\fA}{{\mathfrak{A}}}

\newcommand{\fS}{{\mathfrak{S}}}

\newcommand{\fT}{{\mathfrak{T}}}

\newcommand{\fX}{{\mathfrak{X}}}
\newcommand{\fY}{{\mathfrak{Y}}}

\newcommand{\foral}{\text{ for all }}
\newcommand{\qand}{\quad\text{and}\quad}

\newcommand{\ca}{\mathrm{C}^*}

\newcommand{\cenv}{\mathrm{C}^*_{\textup{env}}}

\newcommand{\ol}{\overline}
\newcommand{\wt}{\widetilde}
\newcommand{\wh}{\widehat}

\newcommand{\alg}{\operatorname{alg}}

\newcommand{\id}{{\operatorname{id}}}

\newcommand{\spn}{\operatorname{span}}

\newcommand{\sca}[1]{\left\langle#1\right\rangle} 
\newcommand{\nor}[1]{\left\Vert #1\right\Vert} 

\newcommand{\tes}[7]{
	\xymatrix@C=2cm@R=1.5cm{
		K_0\left(#1\right) \ar[r]^{#2} & K_0\left(#3\right) \ar[r]^{#4} & K_0\left(#5\right) \ar[d] \\
		K_1\left(#5\right) \ar[u] & K_1\left(#3\right) \ar[l]^{#7} & K_1\left(#1\right) \ar[l]^{#6}
	}
}

\addtocontents{toc}{\protect\setcounter{tocdepth}{1}}

\begin{document}

\title[Hyperrigidity and non-degenerate C*-correspondences]{On hyperrigidity and non-degenerate C*-correspondences}

\author[J.A. Dessi]{Joseph A. Dessi}
\address{School of Mathematics, Statistics and Physics\\ Newcastle University\\ Newcastle upon Tyne\\ NE1 7RU\\ UK}
\email{joseph.dessi@newcastle.ac.uk}

\author[E.T.A. Kakariadis]{Evgenios T.A. Kakariadis}
\address{School of Mathematics, Statistics and Physics\\ Newcastle University\\ Newcastle upon Tyne\\ NE1 7RU\\ UK}
\email{evgenios.kakariadis@newcastle.ac.uk}

\author[I.A. Paraskevas]{Ioannis Apollon Paraskevas}
\address{Department of Mathematics\\ National and Kapodistrian University of Athens\\ Athens\\ 1578 84\\ Greece}
\email{ioparask@math.uoa.gr}



\begin{abstract}
We revisit the results of Kim, and of Katsoulis and Ramsey concerning hyperrigidity for non-degenerate C*-correspondences.
We show that the tensor algebra is hyperrigid, if and only if Katsura's ideal acts non-degenerately, if and only if Katsura's ideal acts non-degenerately under any representation.
This gives a positive answer to the question of Katsoulis and Ramsey, showing that their necessary condition and their sufficient condition for hyperrigidity of the tensor algebra are equivalent.
Non-degeneracy of the left action of Katsura's ideal was also shown by Kim to be equivalent to hyperrigidity for the selfadjoint operator space associated with the C*-correspondence, and our approach provides a simplified proof of this result as well.
In the process we study unitisations of selfadjoint operator spaces in the sense of Werner, and revisit Arveson's criterion connecting maximality with the unique extension property and hyperrigidity, in conjunction with the work of Salomon on generating sets.
\end{abstract}

\maketitle

\section{Introduction}

\subsection{C*-envelope and hyperrigidity}

The C*-envelope of an operator system $\fS$ is the co-universal C*-algebra generated by $\fS$, up to unital complete order embeddings.
Its existence was conjectured by Arveson \cite{Arv69}, and it was proven in full generality by Hamana \cite{Ham79}.
Hamana's conceptual proof passes through the existence of the injective envelope and does not use Dilation Theory.
Almost 30 years later, Dritschel and McCullough \cite{DM05} provided an independent proof of the existence of the C*-envelope for (unital) operator algebras through maximal representations.
The C*-envelope is the C*-algebra generated by a maximal completely isometric representation, i.e., a representation that admits only trivial dilations.

A similar co-universal object has been proven to exist for unital operator spaces, operator systems \cite{Arv08}, possibly non-unital operator algebras \cite{DS18}, and (concrete) selfadjoint operator spaces (i.e., selfadjoint norm-closed subspaces of some $\B(H))$ \cite{CS21, KKM23}, albeit with the right notion of morphisms.
Although the absence of unit for operator algebras can be treated effectively through Meyer's Unitisation Theorem, this is not the case for selfadjoint operator spaces.
Connes and van Suijlekom \cite[Theorem 2.25]{CS21} show that there exists a C*-envelope of a selfadjoint operator space with respect to embeddings, i.e., completely isometric complete order embeddings whose unitisation in the sense of \cite{Wer02} remains completely isometric. 
In \cite[Remark 2.27]{CS21} it is pointed out that embeddings are the right notion of morphisms in order to obtain the existence of a C*-envelope in this category.

Not all $*$-representations of the C*-envelope restrict to maximal representations.
Arveson \cite{Arv11} connected this property to hyperrigidity, a notion that reflects the asymptotic rigidity a $*$-representation may have on the generating set.
In the same work, Arveson provided a comprehensive list of examples of hyperrigid sets.
Duncan \cite{Dun08, Dun10} studied hyperrigidity for big classes of operator algebras coming from dynamical systems or graphs and how it compiles with free products.
A similar study was extended to tensor algebras of C*-correspondences by the second named author \cite{Kak13}.
Dor-On and Salomon \cite{DS18} gave a complete identification for representations of graph tensor algebras to be maximal, while Salomon \cite{Sal17} carefully exhibited rigidity of the zero map with applications to generating (edge) sets of graph C*-algebras.

The full case of non-degenerate C*-correspondences was considered by Katsoulis and Ramsey \cite{KR20, KR21}, and by Kim \cite{Kim21}.
This class encapsulates operator spaces arising from topological graphs and from non-degenerate $*$-endomorphisms.
In \cite{KR21}, as it first appeared on the arXiv, Katsoulis and Ramsey identify non-degeneracy of the left action of Katsura's ideal as the key link with hyperrigidity of the associated tensor algebra, and they show it is sufficient when the C*-correspondence $X$ is countably generated. 
Subsequently, in \cite{Kim21} Kim shows that this condition (without the countably generated assumption on $X$) is not only sufficient, but is also necessary for hyperrigidity of the selfadjoint operator space related to $X$. 
(Note that the published version of \cite{KR21} contains a slight modification of the original argument that makes the proof work for any non-degenerate $X$.)

Our main point of motivation for this work comes from a necessary condition for hyperrigidity identified in \cite{KR20}, concerning a seemingly weaker notion of non-degeneracy, ranging over all representations of the coefficient algebra of $X$.
That this coincides with non-degeneracy of the action of Katsura's ideal for topological graphs was shown to hold when the range map is open in \cite{Kim21}, and without any assumption on the topological graph in \cite{KR20}.
Consequently, Katsoulis and Ramsey \cite{KR20} ask if this is true for \emph{all} non-degenerate C*-correspondences.

There is strong evidence for linking hyperrigidity with weak topological properties.
By comparing the support of the Katsura's ideal with the orthogonal complement of the support of the kernel of the left action in the second dual of the coefficient algebra, Bilich \cite{Bil24} identifies exactly when a representation of the tensor algebra is maximal, recovering the characterisation of hyperrigidity.
This fits in a broader programme around Arveson's hyperrigidity conjecture that has seen a recent boost. 
The conjecture asks whether hyperrigidity amounts to checking maximality just for the boundary representations (i.e., restrictions of irreducible $*$-representations).
Arveson's conjecture was entirely open until recently, where a noncommutative counterexample was provided by Bilich and Dor-On \cite{BD24}.
It remains open in the commutative setting. 
Clou{\^a}tre and Thompson \cite{CT24} further investigated the hyperrigidity conjecture and showed that a modified version holds; in this version the unique extension property is restricted within the weak*-closure of the range of the $*$-representations.
The link of hyperrigidity with the second dual was also identified by Clou{\^a}tre and Saikia \cite{CS23}. 
Motivated by the dilation order established by Davidson and Kennedy \cite{DK24}, the authors in \cite{CS23} prove the existence of a boundary projection in the second dual of the generating C*-algebra that concentrates all dilation maximal states, and whose topological regularity is equivalent to hyperrigidity. 
These works provide strong connections between hyperrigidity and the weak*-setting, indicating that the proposed equivalence by the weak notion of non-degeneracy by Katsoulis and Ramsey \cite{KR20} fits naturally in this context.

\subsection{Non-degenerate C*-correspondences revisited}

In this work we settle the question of Katsoulis and Ramsey \cite{KR20} to the affirmative, as well as obtain an equivalence between the results of \cite{Kim21} and \cite{KR20, KR21} that we hope completes their connections.
Given a C*-correspondence $X$ over a C*-algebra $A$, we write $\T_X$ for the Toeplitz-Pimsner algebra and $\O_X$ for the Cuntz-Pimsner algebra of $X$ with the canonical $*$-epimorphism
\[
\pi_X \times t_X \colon \T_X \to \O_X.
\]
Katsura's ideal is denoted by $J_X$, the left action of $A$ on $X$ is denoted by $\vphi_X$, and the C*-algebra of compact operators on $X$ is denoted by $\K(X)$.
We write $\fS(A,X)$ for the selfadjoint operator space generated by $A$ and $X$ in $\T_X$, and we write $\T_X^+$ for the operator algebra generated by $A$ and $X$ in $\T_X$.
In Theorem \ref{T:hyp} it is summarised that, if $X$ is non-degenerate, then the following items are equivalent:
\begin{enumerate}
\item $\pi_X(A) \cup t_X(X)\subseteq \O_X$ is hyperrigid.
\item $(\pi_X \times t_X)(\fS(A,X)) \subseteq \O_X$ is hyperrigid.
\item $(\pi_X \times t_X)(\T_X^+) \subseteq \O_X$ is hyperrigid.
\item $[\pi(J_X)H] = [\psi_t(\K(X))H]$ for every Cuntz-Pimsner covariant representation $(H,\pi,t)$ of $X$ with $(H,\pi)$ being non-degenerate.
\item $[t(\vphi_X(J_X)X)H] = [t(X)H]$ for every representation $(H,\pi,t)$ of $X$ with $(H,\pi)$ being non-degenerate.
\item $[\vphi_X(J_X)X] \otimes_\si H = X \otimes_\si H$ for every non-degenerate representation $(H,\si)$ of $A$.
\item $[\vphi_X(J_X) \K(X)]  =\K(X)$.
\item $[\vphi_X(J_X) X] = X$.
\item $\vphi_X(J_X) X = X$.
\end{enumerate}
The implication [(vi) $\Rightarrow$ (viii)] settles the question raised in \cite{KR20}.
Our main contribution is the proof of [(iv) $\Leftrightarrow$ (v) $\Leftrightarrow$ (vi) $\Rightarrow$ (vii) $\Leftrightarrow$ (viii)].

In the process of completing the argument we need to revisit some of the proofs of existing results, and let us clarify here the corresponding links.
The implication [(iii) $\Rightarrow$ (vi)] is proven in \cite[Theorem 2.7]{KR20}, and the implication [(viii) $\Rightarrow$ (iii)] is proven in \cite[Theorem 3.1]{KR21}.
We note that the equivalence [(ii) $\Leftrightarrow$ (viii)] is proven in \cite[Theorem 3.5]{Kim21}; however, herein we get this by showing the equivalences [(i) $\Leftrightarrow$ (ii) $\Leftrightarrow$ (iii) $\Leftrightarrow$ (viii)] using that $\fS(A,X)$ and $\T_X^+$ have the same set of generators and that the representations of $\T_X^+$ satisfy a specific invariance property noted in \cite{KK12}.
Moreover, a general C*-result is used for the equivalence [(iv) $\Leftrightarrow$ (vii)], thus avoiding the ultraproduct argument in \cite[proof of Theorem 3.5]{Kim21}, see Proposition \ref{P:multiplier}.

In passing, in Proposition \ref{P:cenv S(A,X)} we clarify that the map $\fS(A,X) \hookrightarrow \O_X$ is an embedding in the theory of selfadjoint operator spaces; this fills a gap in the proof of \cite[Proposition 2.6]{Kim21} and ascertains that $\O_X$ is the C*-envelope of $\fS(A,X)$ in the sense of \cite{CS21, KKM23}.
This is a subtle point, as C*-inclusions of selfadjoint operator spaces may not be embeddings. 
In particular building on \cite[Example 2.14]{KKM23} and \cite[Example 4.1]{Rus23}, in Example \ref{E:not incl} we show that there exists an inclusion of a selfadjoint operator space that is not an embedding.
In Corollary \ref{C:ext} we provide a characterisation for an inclusion to be an embedding connecting it with the extensions of completely contractive completely positive functionals. 
In Proposition \ref{P:ext} we show that for a class of selfadjoint operator spaces that contain a C*-algebra and an approximate unit any inclusion is automatically an embedding. This class of examples contains stabilisations of selfadjoint operator spaces in the sense of \cite{CS21} and also selfadjoint operator spaces arising from non-degenerate C*-correspondences.

\subsection{Hyperrigidity and maximality revisited}

Arveson's maximality criterion states that an operator system $\fS$ is hyperrigid if and only if the restriction of any unital $*$-representation of its C*-envelope to $\fS$ is maximal, thus making detection of hyperrigidity amenable to $*$-algebraic manipulations.
This approach has been used with much success in several settings and extends directly to generating sets when they contain the unit.
However moving beyond that point requires more care as there may be non-zero unital completely positive maps annihilating a generating set.
Such maps are shown to exist in \cite[Example 3.4]{Sal17} for the set of generators $\{S_1, \dots, S_d\}$ of the Cuntz algebra $\O_d$.
More generally, in his influential work Salomon \cite{Sal17} shows that this is the case for edge sets in C*-algebras of row-finite graphs.

In order to encounter this phenomenon, Salomon \cite{Sal17} carefully studied the compatibility of hyperrigidity under direct limits, and reduced the problem to two main classes: when $\ca(\G)$ is unital and when the zero map on $\G$ has a unique extension (i.e., it is \emph{rigid}).
When $\ca(\G)$ is unital, in \cite[Proposition 2.1]{Sal17} it is shown that the non-degeneracy clause in the definition of hyperrigidity allows to connect it still with maximality by adding the unit to the generating set; this is the case for $\{S_1, \dots, S_d\}$ in $\O_d$. 
Hyperrigidity of a generating set $\G$ with $\ca(\G)$ being unital is also studied in \cite{PS25}.
On the other hand, in \cite[Theorem 3.3]{Sal17} it is shown that rigidity of the zero map is equivalent to the states of $\ca(\G)$ not vanishing on $\G$, and from there the connection between hyperrigidity and maximality is achieved.
We refer to this property as $\G$ being \emph{separating}.
This is the case for non-degenerate C*-correspondences where $\pi_X(A) \cup t_X(X)\subseteq \O_X$ is a separating generating set. Moreover, in Remark \ref{R:septensor} we show that if $\pi_X(A) \cup t_X(X)$ is hyperrigid in $\O_X$, then $X$ is non-degenerate if and only if $\pi_X(A) \cup t_X(X)$ is separating.

The connection between hyperrigidity and maximality in \cite{Sal17} makes use of \cite[Lemma 2.7, item (ii)]{Sal17} which refers to a general $\G$; however in Remark \ref{R:Salunit} we show that this argument applies only in the unital context.
Nevertheless, the statement is still valid when $\G$ is separating, and we provide a proof to this end.
It is unclear to us whether \cite[Lemma 2.7, item (ii)]{Sal17} holds when $\ca(\G)$ is not unital and $\G$ is not separating (there is an abundance of such hyperrigid sets exhibited in \cite{Sal17}).
In Theorem \ref{T:hypiffuepun} and Theorem \ref{T:hypiffuep} it is summarised that, if $\ca(\G)$ is unital or if $\G$ is separating in $\ca(\G)$, then the following are equivalent:
\begin{enumerate}
\item $\G$ is hyperrigid.
\item Every non-degenerate $*$-representation $\Phi$ of $\ca(\G)$ has the unique extension property with respect to $\G$.
\item Every non-degenerate $*$-representation $\Phi$ of $\ca(\G)$ is maximal on $\G$.
\end{enumerate}

When $\G$ is separating the statements (ii) and (iii) above are still equivalent to $\G$ being hyperrigid without the non-degeneracy clause.
Conversely, removing the non-degeneracy clause a priori imposes that $\G$ is separating.
Therefore we insist on working within the class of non-degenerate $*$-representations; this has an effect on the unique extension property variant we use in the non-unital context which amounts for different (yet still non-degenerate) $*$-representations on the domain.
Aligning with \cite{Sal17}, in Proposition \ref{P:unit ccp/ucp} we verify that $\G$ is separating and hyperrigid if and only if the augmented set in the unitisation by the unit is hyperrigid, when $\ca(\G)$ is not unital.
As shown in \cite{Sal17} there are hyperrigid non-separating sets $\G$ with non-unital $\ca(\G)$.
Finally, we study a variant of hyperrigidity with respect to completely contractive completely positive maps, motivated by the setting of \cite[Lemma 2.7]{Sal17}.
In Corollary \ref{C:ccp/ucp} we show that it coincides with Arveson's hyperrigidity, without any assumptions on $\G$.

\subsection{Contents}

In this work we revise and update existing results from the literature, providing corrections and modifications where needed. 
We will be crediting the original papers, even if the proofs are corrected or completed here.
The structure of the manuscript is as follows.

In Section \ref{S:pre} we record the main definitions from the Operator Spaces theory, and we provide a criterion for an inclusion of a class of selfadjoint operator spaces to be an embedding.
This is required in order to complete the proof of Kim \cite[Proposition 2.6]{Kim21} that $\O_X$ is a C*-extension in the sense of Connes and van Suijlekom \cite{CS21}, and thus the C*-envelope of $\fS(A,X)$.
Building on \cite[Example 2.14]{KKM23} and \cite[Example 4.1]{Rus23} we provide an example of an inclusion of a selfadjoint operator space inside the C*-algebra it generates that is \emph{not} an embedding. Together with Chatzinikolaou, we will expand on embeddings of selfadjoint operator spaces in a future work \cite{CDKP24}.

In Section \ref{S:hyper} we study the connection of hyperrigidity with maximality when $\ca(\G)$ is unital or when the generating set $\G$ is separating.
We complete the argument of \cite[Lemma 2.7, item (ii)]{Sal17} for the separating case in Theorem \ref{T:hypiffuep}.
In Remark \ref{R:Salunit} we show why the strategy outlined for \cite[Lemma 2.7, item (ii)]{Sal17} does not apply in the non-unital context.
This relies on Proposition \ref{P:multiplier}, where we show that if $\A \subseteq \B \subseteq \C$ are C*-algebras, then $[\A \B] = \B$ if and only if $[\Phi(\A)H] = [\Phi(\B)H]$ for every $*$-representation $\Phi \colon \C \to \B(H)$.
Throughout the section we comment on the equivalence between the ucp and the ccp variants of hyperrigidity and the unique extension property.
This relies on Proposition \ref{P:Dil} where we apply Stinespring's Dilation Theorem on Werner's unitisation when the map and/or the C*-algebra is not unital.

In Section \ref{S:cor} we answer the question of \cite{KR20} and provide further equivalences for hyperrigidity, as well as the equivalence with the results of \cite{Kim21}.
We note that the case of a separating $\G$ is exactly where the class of non-degenerate C*-correspondences sits. 
Conversely, in Remark \ref{R:septensor} we point out that a C*-correspondence with hyperrigid and separating tensor algebra is automatically non-degenerate.
\subsection*{Acknowledgements}
Evgenios Kakariadis and Apollonas Paraskevas acknowledge that this research work was supported within the framework of the National Recovery and Resilience Plan Greece 2.0, funded by the European Union - NextGenerationEU (Implementation Body: HFRI. Project name: Noncommutative Analysis: Operator Systems and Nonlocality.
HFRI Project Number: 015825).
Apollonas Paraskevas acknowledges that this research work was supported by the Hellenic Foundation
for Research and Innovation (HFRI) under the 5th Call for HFRI PhD Fellowships (Fellowship Number: 19145).

The authors would like to thank Se-Jin (Sam) Kim and Guy Salomon for their helpful comments and remarks on the draft of the manuscript. The authors would like to thank the referee for the comments and remarks that helped improve the presentation.

\begin{open}
For the purpose of open access, the authors have applied a Creative Commons Attribution (CC BY) license to any Author Accepted Manuscript (AAM) version arising.
\end{open}
\section{Preliminaries} \label{S:pre}

\subsection{Operator Spaces}

If $X$ is a subset of a Banach space $Y$, then we write $[X]$ for the closed linear span of $X$ in $Y$.
We will reserve the letters, $H, K, L$ etc.\ to denote Hilbert spaces.

By an operator space $\fX$ we will mean a norm-closed subspace of some $\B(H)$.
By an operator algebra $\fA$ we will mean a norm-closed subalgebra of some $\B(H)$.
Every operator algebra attains a C*-cover, i.e., a completely isometric homomorphism $\io \colon \fA \to \C$ such that $\C = \ca(\io(\fA))$.
The \emph{C*-envelope} $\cenv(\fA)$ is the co-universal C*-cover of $\fA$, i.e., there exists a completely isometric homomorphism $\io \colon \fA \to \cenv(\fA)$ such that for any other C*-cover $j \colon \fA \to \C$ there exists a unique $*$-epimorphism $\Phi \colon \C \to \cenv(\fA)$ such that $\Phi \circ j = \io$.
The existence of the C*-envelope for a unital operator algebra is established through its injective envelope, while the C*-envelope for a non-unital operator algebra is the C*-algebra it generates in the C*-envelope of its unitisation.
Note that $\fA$ is unital if and only if there is a completely isometric homomorphism $\phi \colon \fA \to \B(H)$ with $I_H \in [\phi(\fA)]$.
The reader is addressed to \cite{BL04, Pau02} for the full details.

By an operator system $\fS$ we will mean a unital selfadjoint norm-closed subspace of some $\B(H)$.
The notion of the C*-envelope passes to this category in a natural way, up to unital completely isometric maps.
A linear map $\phi\colon \fS\to \fT$ where $\fS$ and $\fT$ are operator systems is called a (unital) \emph{complete order isomorphism} if it is invertible and both $\phi$ and $\phi^{-1}$ are (unital) completely positive. 
We also say that $\phi \colon \fS \to \fT$ is a \emph{complete order embedding} if it is a complete order isomorphism onto its image. 
Note that, if $\phi$ is unital, then it is a complete isometry if and only if it is a complete order embedding. 
We also have that, if a linear map between operator systems is a completely isometric complete order isomorphism, then it is necessarily unital.

The existence of the C*-envelope of an operator system was established by Hamana \cite{Ham79} through the existence of the injective envelope.
An independent proof for unital operator algebras was established by Dritschel and McCullough \cite{DM05} through the existence of maximal dilations.
Recall that a map $\phi \colon \fX \to \B(K)$ on an operator space $\fX$ is called a \emph{dilation} of a map $\io \colon \fX \to \B(H)$ if $H \subseteq K$ and $\io = P_H \phi|_H$.
A map is called \emph{maximal} if it only attains trivial dilations, i.e., dilations by orthogonal summands.
Dritschel and McCullough \cite{DM05} show that every unital completely contractive homomorphism of a unital operator algebra $\fA$ admits a dilation that is a maximal completely contractive homomorphism; then the C*-envelope of $\fA$ is the C*-algebra generated by a unital completely isometric maximal representation.
A simplified proof by Arveson \cite{Arv08} yields the same result for operator systems. 

The same result on maximal dilations for non-unital operator algebras follows by combining \cite[Proposition 2.4]{DS18} of Dor-On and Salomon with \cite[Proposition 4.8]{Sal17} of Salomon.
One of the important steps in this respect is Meyer's Unitisation Theorem which asserts that for an operator algebra $\fA \subseteq \B(L)$ with $I_L \notin \fA$ and a completely contractive (resp.\ completely isometric) homomorphism $\phi \colon \fA \to \B(H)$, the (unique) unital extension of $\phi$ on $\fA +\bC I_L$ is a completely contractive (resp.\ completely isometric) homomorphism.
Thus the unitisation is unique up to completely isometric representations.

One of the pivotal steps in Arveson's work \cite{Arv08} is to connect maximality with the unique extension property of a map.
A completely contractive map $\phi \colon \fX \to \B(H)$ on an operator space $\fX \subseteq \ca(\fX)$ is said to have \emph{the unique extension property} if:
\begin{enumerate}
\item $\phi$ has a unique completely contractive completely positive extension $\wt{\phi} \colon \ca(\fX) \to \B(H)$, and
\item $\wt{\phi}$ is a $*$-representation.
\end{enumerate}
Arveson \cite{Arv08} shows that a unital completely positive map on an operator system has the unique extension property if and only if it is maximal.
The same holds for a unital completely contractive representation on a unital operator algebra.
By using Meyer's Unitisation Theorem, Dor-On and Salomon \cite{DS18} show that the same holds for any operator algebra.

\subsection{Selfadjoint Operator Spaces}

We next present some elements of the theory of selfadjoint operator spaces in the sense of \cite{CS21, KKM23, Wer02}. These spaces also appear as (non-unital) operator systems in the literature. 
However, in \cite{Wer02} operator systems refer to matrix ordered operator spaces that are completely order embedded into some $\B(H)$, in contrast to \cite{CS21, KKM23} where the complete order embedding is also assumed to be completely isometric. 
Since we will restrict ourselves to the concrete case and we want to reflect more closely the possible lack of the unit we adopt the following terminology: a \emph{selfadjoint operator space} will be a selfadjoint norm-closed subspace of some $\B(H)$.

Werner \cite[Definition 4.7]{Wer02} considers the following notion of unitisation for a selfadjoint operator space $\fS \subseteq \B(H)$.
Consider the space $\fS^{\#} := \fS \oplus \bC$ with an involution given by $(s,\al)^*=(s^*,\ol{\al})$.
Declare a selfadjoint element $(s,\al) \in M_n(\fS^{\#})$ to be \emph{positive} if $\al \geq 0$ and 
\[
\vphi(\al_\varepsilon^{-1/2}s\al_\varepsilon^{-1/2})\geq -1 \foral \eps>0 \text{ and } \vphi \in {\rm CCP}(M_n(\fS), \bC),
\]
where $\al_\varepsilon:=\al+\varepsilon I_{M_n}$.
In \cite[Lemma 4.8]{Wer02} Werner shows that $\fS^{\#}$ equipped with this matrix ordering is an operator system, and that the inclusion map $\fS \hookrightarrow \fS^{\#}$ is completely isometric and a complete order embedding.
It follows that Werner's unitisation of a C*-algebra (as a selfadjoint operator space) coincides with the usual unitisation \cite[Corollary 4.17]{Wer02}.

A map $\phi \colon \fS \to \fT$ between selfadjoint operator spaces $\fS$ and $\fT$ is called \emph{positive} if it is selfadjoint and maps positive elements to positive elements, as the cones do not necessarily span the spaces.
Such a map $\phi \colon \fS \to \fT$ is called a \emph{complete order embedding} if it is completely positive and invertible onto its image with a completely positive inverse. 
If such a $\phi$ is also surjective we say that it is a \emph{complete order isomorphism}.
Whenever $\phi$ is a completely isometric complete order isomorphism we have that $\phi^\#$ is a unital complete order isomorphism of operator systems, by \cite[Corollary 2.14]{CS21}. 
We say that $\phi\colon \fS \to \fT$ is an \emph{embedding} if it is a completely isometric complete order embedding, and the unitisation map
\[
\phi^\#\colon \fS^\#\to\fT^\#; (s,\la)\mapsto (\phi(s),\la)
\]
is a complete order embedding of operator systems.
More generally, by \cite[Lemma 4.9]{Wer02}, we have that if $\phi\colon \fS \to \fT$ is a completely contractive completely positive map, then the natural extension $\phi^{\#} \colon \fS^{\#} \to \fT^{\#}$ is unital and completely positive.
Throughout this text, by an abuse of notation, for a map $\phi\colon \fS\to \fT$ where $\fS$ is a selfadjoint operator space and $\fT$ is a (unital) operator system we will be writing $\phi^\#\colon \fS^\#\to \fT$ also for the map such that $\phi^\#(s,\la)=\phi(s)+\la 1_\fT$.

\begin{remark} \label{R:c* unit}
We will be making the following convention for a completely contractive completely positive map $\phi \colon \C \to \B(H)$ on a C*-algebra $\C$.
If $\C$ is not unital, then we extend $\phi$ to its unitisation $\phi^{\#} \colon \C^{\#} \to \B(H)$ by considering 
\[
\phi^{\#}(c,\la) = \phi(c) + \la I_H,
\]
which is unital completely positive (and thus completely contractive), see \cite[Proposition 2.2.1]{BO08}.
If $\C$ has a unit $1$, but $\phi(1) \neq I_H$, then again we consider the same formula for $\phi^{\#} \colon \C^{\#} \to \B(H)$ which is unital completely positive (and thus completely contractive).
However, in this case $\C^{\#}$ is $*$-isomorphic as a C*-algebra to $\C \oplus \bC(1^{\#} - (1,0))$ for the unit $1^{\#}:= (0, 1)$ of $\C^{\#}$, see \cite[Lemma 2.2.3]{BO08}.

Another way to see this is by considering a faithful non-degenerate $*$-representation $\C \subseteq \B(K)$.
If $\C$ is not unital, then $I_K \notin \C$; we write $\C^{\#} := \C + \bC I_K$ and set 
\[
\phi^{\#}(c +\la I_K) = \phi(c) + \la I_H.
\]
If $I_K \in \C$, but $\phi(I_K)\neq I_H$, then we consider $\B(K \oplus \bC)$ and $\C$ as its $(1,1)$-corner; in this case we write $\C^{\#} := \C + 1_{K \oplus \bC}$ and it follows that $\C^{\#} = \C \oplus \bC p$ for $p=I_{K \oplus \bC} - I_K$.
We then define
\[
\phi^{\#}(c + \la p) = \phi(c) + \la (I_H - \phi(1_K)),
\]
which is unital completely positive (and thus completely contractive).
\end{remark}

In \cite[Theorem 2.25]{CS21} Connes and van Suijlekom prove that there exists a (necessarily unique up to $*$-isomorphisms) C*-envelope in the category of selfadjoint operator spaces with respect to embeddings, rather than completely isometric complete order embeddings.
To make a distinction, the authors in \cite{CS21} refer to it as the ${\rm C}^{\#}$-envelope; however the authors in \cite{KKM23} refer to it as the C*-envelope.
We will use the latter terminology.

A \emph{C*-extension} of a selfadjoint operator space $\fS$ is a pair $(\C, \io)$ where $\io \colon \fS \to \C$ is an embedding and $\ca(\io(\fS)) = \C$.
In \cite[Lemma 2.22]{CS21} it is shown that any $\fS$ admits a unique minimal C*-extension $\cenv(\fS)$, which we will call \emph{the C*-envelope of $\fS$}.
In particular, $\cenv(\fS)$ is the C*-algebra generated by $\fS$ inside $\cenv(\fS^{\#})$.
We should note that in \cite[Theorem 2.25]{CS21} the authors also prove that, if $\fS$ is an operator system, then the C*-envelope of $\fS$ as a selfadjoint operator space coincides with the C*-envelope of $\fS$ as an operator system.

The notion of an embedding in the category of selfadjoint operator spaces is the suitable one in order to obtain the existence of a minimal C*-cover, see \cite[Remark 2.27]{CS21}. 
The requirement of $\phi^\#$ being an embedding of operator systems is justified also by \cite[Section 4]{Rus23} and \cite[Example 2.14]{KKM23}, where it is emphasised that inclusions of selfadjoint operator spaces are not automatically embeddings.
It is thus natural to ask, when the inclusion of a selfadjoint operator space in the C*-algebra it generates is an embedding.
We have the following \lq\lq natural'' lemma.

\begin{lemma} \label{L:ext}
Let $\fS$ and $\fT$ be selfadjoint operator spaces and let $\phi \colon \fS \to \fT$ be a completely isometric complete order embedding.
Then the map $\phi$ is an embedding if and only if every $\vphi \in {\rm CCP}(M_n(\phi(\fS)), \bC)$ extends to a $\wt{\vphi} \in {\rm CCP}(M_n(\fT), \bC)$, for all $n \in \bN$.
\end{lemma}

\begin{proof}
Suppose that $\phi$ is an embedding, so that $\phi^\#$ is a unital complete order embedding.
Hence $M_n(\phi(\fS)^\#)$ embeds in $M_n(\fT^\#)$ for every $n\in\bN$.
By \cite[Corollary 4.4]{Wer02} we thus obtain the inclusions
\[
M_n(\phi(\fS))^\# \subseteq M_n(\phi(\fS)^\#) \subseteq M_n(\fT^\#)
\foral 
n \in \bN.
\]
Now let $\vphi$ be in ${\rm CCP}(M_n(\phi(\fS)), \bC)$ and consider its unitisation $\vphi^\# \in {\rm UCP}(M_n(\phi(\fS))^\#, \bC)$.
By Arveson's Extension Theorem we obtain an extension $\wt{\vphi^\#} \in {\rm UCP}(M_n(\fT^\#), \bC)$ of $\vphi^\#$.
Then the map 
\[
\wt{\vphi} := \wt{\vphi^\#}|_{M_n(\fT)}
\]
is the required completely contractive and completely positive extension of $\vphi$.

Conversely, take $(\phi^{(n)}(s), \al)$ be positive in $M_n(\fT^\#)$.
Then $\al \geq 0$.
Take $\vphi \in {\rm CCP}(M_n(\fS),\bC)$ and consider 
\[
\vphi \circ (\phi^{(n)})^{-1} \in {\rm CCP}(M_n(\phi(\fS)), \bC).
\]
By assumption we have that it extends to a $\wt{\vphi} \in {\rm CCP}(M_n(\fT), \bC)$ and thus
\[
\wt{\vphi}(\al_{\eps}^{-1/2} \phi^{(n)}(s) \al_{\eps}^{-1/2})
\geq
-1
\foral 
\eps>0.
\]
But then we get
\[
\vphi(\al_{\eps}^{-1/2} s \al_{\eps}^{-1/2})
=
\wt{\vphi}(\al_{\eps}^{-1/2} \phi^{(n)}(s) \al_{\eps}^{-1/2})
\geq
-1
\foral 
\eps>0,
\]
where we use the fact that $\phi$ is a $\bC$-bimodule map.
Hence $(s,\al)$ is positive in $M_n(\fS^\#)$, as required.
\end{proof}

Under the assumptions of Lemma \ref{L:ext}, we see that if $\wt{\vphi} \in {\rm CCP}(M_n(\fT), \bC)$ is an extension of $\vphi \in {\rm CCP}(M_n(\phi(\fS)), \bC)$ with $\| \vphi \| = 1$, then $\|\wt{\vphi}\| = 1$ as well.
The following corollary is an immediate application of Lemma \ref{L:ext}.

\begin{corollary} \label{C:ext}
Let $\fS \subseteq \ca(\fS)$ be a selfadjoint operator space. 
The following are equivalent:
\begin{enumerate}
\item The inclusion $\fS \subseteq \ca(\fS)$ is an embedding.
\item For every $n\in \bN$ and every positive map $\vphi\colon M_n(\fS)\to \bC$ with $\|\vphi\|=1$ there exists a positive map $\wt{\vphi}\colon M_n(\ca(\fS))\to \bC$ with $\|\wt{\vphi}\|=1$ that extends $\vphi$. 
\end{enumerate}
In particular, if any (and thus all) of the above is satisfied, then $\fS^\#$ is completely order isomorphic with the operator system $\fS+\bC 1^\#\subseteq \ca(\fS)^{\#}$, with a canonical identification of the units.
\end{corollary}

By using the setup of \cite[Example 2.14]{KKM23} and \cite[Example 4.1]{Rus23}, we get an example of an inclusion $\fS \subseteq \ca(\fS)$ that is not an embedding.

\begin{example} \label{E:not incl}
As in \cite[Example 2.14]{KKM23}, consider the element $x := E_{11} - \frac{1}{2}E_{22}$ in $M_2$.
It then follows that the space generated by $x$ in $M_2$ is a (one-dimensional) selfadjoint operator space with trivial cone structure, i.e., all matricial cones are $\{0\}$, as inherited from intersecting with the cones of $M_2$.
Moreover, it is easy to see that $\ca(\{x\}) = \bC^2$, which has the usual matricial cone structure in $M_2$.

We claim that the inclusion $\bC x \subseteq \ca(\{x\})$ is not an embedding as $\bC x$ admits completely contractive completely positive functionals that have no completely contractive completely positive extensions on $\bC^2$.
For such an example take $\vphi$ such that $\vphi(x) = -1$.
Then $\vphi$ is positive on $\bC x$ as the cone structure is trivial, and of norm one as it is minus the compression to the $(1,1)$-entry.
However $\vphi$ does not extend to a norm one positive functional on $\bC^2$ and thus on $M_2$.

To reach a contradiction, suppose there is a norm one positive functional $\wt{\vphi} \colon M_2 \to \bC$ such that $\wt{\vphi}(x) = \vphi(x) = -1$; then $\wt{\vphi}(\cdot) = {\rm Tr}(\varrho \cdot)$ for a positive matrix $\varrho \in M_2$ with unital trace, i.e., $\varrho_{11} + \varrho_{22} = 1$.
By applying $\wt{\vphi}(x) = -1$ we get $\varrho_{11} - \frac{1}{2}\varrho_{22} = -1$.
Hence we have
\[
2\varrho_{11} + \frac{1}{2}\varrho_{22} = \left(\varrho_{11} + \varrho_{22}\right) + \bigg(\varrho_{11} - \frac{1}{2}\varrho_{22}\bigg) = 1+(-1)= 0.
\]
Hence $\varrho_{11} = \varrho_{22} = 0$ (as $\varrho$ is positive), which is a contradiction.

More generally one can consider for fixed $n \geq 2$ the element $x_n:= E_{11} - \sum_{k=2}^n \frac{1}{k} E_{kk}$.
The same arguments apply as $\ca(\{x_n\}) = \bC^n$ inside $M_n$.
Note that the selfadjoint operator spaces $\bC x_n$ are all completely order isomorphic for any $n \geq 2$, since they all have trivial cone structure, and none is complete order isomorphic to $\bC$ (as the latter has non-trivial cones).
\end{example}

In what follows we identify a class of selfadjoint operator spaces for which the inclusion $\fS \subseteq \ca(\fS)$ is automatically an embedding. 

\begin{proposition}\label{P:ext}
Let $\fS \subseteq \ca(\fS)$ be a selfadjoint operator space. 
Suppose there is a C*-algebra $A \subseteq \fS$ such that $A \cdot \fS \subseteq \fS$, and an approximate unit $(u_\la)_\la$ for $A$ such that $\lim_\la u_\la s = s$ for every $s\in \fS$. 
Then every $\vphi \in {\rm CCP}(M_n(\fS), \bC)$ attains an extension to a $\wt{\vphi} \in {\rm CCP}(M_n(\ca(\fS)), \bC)$ with $\nor{\wt{\vphi}} = \nor{\vphi}$, for any $n \in \bN$.
Consequently, the inclusion $\fS \subseteq \ca(\fS)$ is an embedding.
\end{proposition}

\begin{proof}
We note that the last conclusion follows because of Corollary \ref{C:ext}.
Hence it suffices to show the extension property.

Towards this end, we see that since $A$ and $\fS$ are selfadjoint, and since every $u_\la$ is positive, then we have $\fS \cdot A \subseteq \fS$ and $\lim_\la s u_\la = s$ for all $s \in \fS$.
Likewise for $(u_\la \otimes I_n)_\la$ in $M_n(\fS)$, for any $n \in \bN$.
For notational convenience set $u_\la^{(n)}:= u_\la \otimes I_{n}$ and note that 
\[
\lim_\la u_\la^{(n)} c = c = \lim_\la c u_\la^{(n)}
\foral c\in M_n(\ca(\fS)),
\]
since $\fS$ generates $\ca(\fS)$.

We need to show that every $\vphi \in {\rm CCP}(M_n(\fS), \bC)$ admits a contractive positive extension $\wt{\vphi} \colon M_n(\ca(\fS))\to \bC$ (note that $\wt{\vphi}$ will then be automatically completely contractive and completely positive). 
Without loss of generality we may assume that $\|\vphi\|=1$.
We claim that it suffices to show that $\lim_\la \vphi(u_\la^{(n)})=1$.

Indeed, if that is the case, then by the Hahn-Banach Theorem we may extend $\vphi$ to a functional $\wt{\vphi}\colon M_n(\ca(\fS))\to \bC$ satisfying $\|\wt{\vphi}\|= 1$. 
Since $(u_\la^{(n)})_\la$ is also an approximate unit for $M_n(\ca(\fS))$ satisfying 
\[
\|\wt{\vphi}\| = 1 = \lim_\la \vphi(u_\la^{(n)}) = \lim_\la\wt{\vphi}(u_\la^{(n)}),
\]
we will then get that $\wt{\vphi}$ is positive, see \cite[Theorem 3.3.3]{Mur90}.

We now show that $\lim_\la \vphi(u_\la^{(n)})=1$ by amending the C*-argument of \cite[Theorem 3.3.3]{Mur90}.
Fix $s\in M_n(\fS)$ such that $\|s\|\leq 1$ and let $\al \in \bC$. 
Then, in $\ca(\fS)$, we have
\begin{align*}
0 
& \leq
(\al s u_{\mu}^{(n)}-u_\la^{(n)})^* (\al s u_{\mu}^{(n)}-u_\la^{(n)}) \\
& =
|\al|^2 u_\mu^{(n)} s^*su_\mu^{(n)}-\ol{\al}u_\mu^{(n)} s^* u_\la^{(n)}-\al u_\la^{(n)} s u_\mu^{(n)} +(u_\la^{(n)})^2\\
& \leq 
|\al|^2 u_\mu^{(n)} u_\mu^{(n)}-\ol{\al}u_\mu^{(n)} s^* u_\la^{(n)}-\al u_\la^{(n)} s u_\mu^{(n)} +(u_\la^{(n)})^2,
\end{align*}
where we used that
\[
u_\mu^{(n)} s^*su_\mu^{(n)} \leq \|s\|^2 u_\mu^{(n)} u_\mu^{(n)} \leq u_\mu^{(n)} u_\mu^{(n)}.
\]
Since $\vphi$ is positive and $\vphi(u_\mu^{(n)} u_\mu^{(n)}) \leq \|\vphi\|=1$, we obtain
\begin{align*}
0
& \leq 
|\al|^2 \vphi(u_\mu^{(n)} u_\mu^{(n)})-\ol{\al}\vphi(u_\mu^{(n)} s^* u_\la^{(n)})-\al \vphi(u_\la^{(n)} s u_\mu^{(n)}) + \vphi((u_\la^{(n)})^2)\\
& \leq
|\al|^2-\ol{\al}\vphi(u_\mu^{(n)} s^* u_\la^{(n)})-\al \vphi(u_\la^{(n)} s u_\mu^{(n)}) + \vphi((u_\la^{(n)})^2).
\end{align*}
Taking the limit with respect to $\mu$ yields
\[
0 \leq |\al|^2-\ol{\al}\vphi(s^* u_\la^{(n)})-\al \vphi(u_\la^{(n)} s ) + \vphi((u_\la^{(n)})^2), \foral \la.
\]
By setting $\al = \vphi(s^*u_\la^{(n)})$, and using that $\nor{s} \leq 1$ and $u_\la^2 \leq u_\la$ for all $\la$, we get
\[
|\vphi(u_\la^{(n)} s)|^2
\leq 
\vphi((u_\la^{(n)})^2) \leq \vphi(u_\la^{(n)}).
\]
Hence taking the limit inferior with respect to $\la$ implies 
\[
|\vphi(s)|^2 = \liminf_\la |\vphi(u_\la^{(n)} s)|^2 \leq \liminf_\la \vphi(u_\la^{(n)}) \leq \|\vphi\| = 1.
\] 
Taking the supremum over $\|s\| \leq 1$ gives 
\[
1=\liminf_\la \vphi(u_\la^{(n)})\leq \limsup_\la \vphi(u_\la^{(n)})\leq \|\vphi\|=1,
\]
and hence we obtain $\lim_\la \vphi(u_\la^{(n)})=1$, as required.
\end{proof}

\begin{remark}
A class of (non-unital) selfadjoint operator spaces that satisfies the assumptions of Proposition \ref{P:ext} is given by the stabilisations of operator systems considered by Connes and van Suijlekom \cite{CS21}.
Indeed if $\S$ is an operator system and $\K$ denotes the compact operators on $\ell^2$, then $1_\S \otimes \K$ forms a C*-subalgebra of $\S \otimes \K$ that satisfies these assumptions.
As we will see later a second class arises in the context of non-degenerate C*-correspondences.
\end{remark}

\subsection{C*-algebras}

We close this section with some general C*-results that we will need later on.
Passing to the unitisation of a C*-algebra allows to use Stinespring's Dilation Theorem in the strong sense of corners instead of just conjugations of $*$-representations by a contraction.

\begin{proposition}\label{P:Dil}
Let $\C$ be a C*-algebra and $\phi \colon \C \to \B(H)$ be a completely contractive and completely positive map.
Then there exists a Hilbert space $H' \supseteq H$ and a $*$-homomorphism $\Phi \colon \C \to \B(H')$ such that $\phi = P_H \Phi |_H$.
\end{proposition}

\begin{proof}
Using the convention of Remark \ref{R:c* unit}, by \cite[Corollary 4.17 and Proposition 4.9 (c)]{Wer02} we have that the map
\[
\phi^{\#}\colon \C^{\#} \to \B(H)\text{; } c + \la 1^{\#}  \mapsto \phi(c) + \la I_H,
\]
is unital and completely positive.
By Stinespring's Dilation Theorem, the unital map $\phi^{\#}$ dilates to a $*$-representation $\Phi^{\#}\colon \ca(\G)^\#\to \B(H')$ where $H'\supseteq H$ so that $\phi^{\#} = P_H \Phi^{\#} |_{H}$. 
Setting $\Phi := \Phi^{\#}|_{\C}$ completes the proof.
\end{proof}

We will also need the following proposition.

\begin{proposition}\label{P:multiplier}
Let $\A \subseteq \B \subseteq \C$ be C*-algebras.
The following are equivalent:
\begin{enumerate}
\item $[\A \B] = \B$;
\item $[\B \A] = \B$;
\item $[\Phi(\A) H] = [\Phi(\B) H]$ for every non-degenerate $*$-representation $\Phi \colon \C \to \B(H)$ ;
\item $[\Phi(\A) H] = [\Phi(\B) H]$ for every $*$-representation $\Phi \colon \C \to \B(H)$.
\end{enumerate}

If in addition $\B$ is unital, then any (and thus all) of the items above is equivalent to $\A$ containing the unit of $\B$.
\end{proposition}

\begin{proof}
The implication [(iv) $\Rightarrow$ (iii)] is immediate, while [(i) $\Leftrightarrow$ (ii)] follows by taking adjoints.
If item (i) holds, then we have
\[
[\Phi(\A) H] \subseteq [\Phi(\B) H] = [\Phi(\A \B) H] = [\Phi(\A) \Phi(\B) H] \subseteq [\Phi(\A) H]
\]
for any $*$-representation $\Phi \colon \C \to \B(H)$, and thus item (iv) follows. 

Next we prove the implication [(iii) $\Rightarrow$ (ii)]. Assume that $[\B \A]$ is a proper left ideal of $\B$.
Then there exists a state $\tau$ of $\B$ such that $[\B \A] \subseteq N_\tau$, and we extend $\tau$ to a state of $\C$ which we denote by the same symbol, see for example \cite[Theorem 5.3.3 and Theorem 3.3.8]{Mur90}.
Consider the GNS-representation $(H_\tau, \Phi_\tau, \xi_\tau)$ of $\tau$, where $\xi_\tau$ is the canonical cyclic vector obtained as in \cite[Theorem 5.1.1]{Mur90}.
We will show that for the non-degenerate $*$-representation $\Phi_\tau$ we have
\[
[\Phi_\tau(\A) H_\tau] \subsetneq [\Phi_\tau(\B) H_{\tau}],
\]
which completes the proof of this implication.

Towards this end, let $P_\A$ be the projection on $[\Phi_\tau(\A) H_\tau]$ and $P_\B$ be the projection on $[\Phi_\tau(\B) H_\tau]$.
Fix $(a_\la)_\la$ be an approximate unit of $\A$, and $(b_\mu)_\mu$ be an approximate unit of $\B$, for which we get
\[
P_\A = \text{w*-}\lim_\la \Phi_{\tau}(a_\la)
\qand
P_\B = \text{w*-}\lim_\mu \Phi_{\tau}(b_\mu).
\]
Since $\A \subseteq \B$ we have $P_\A \leq P_\B$, and thus $P_\A P_\B = P_\B P_\A = P_\A$.
Consider the vector $h:= P_\B \xi_\tau$.
Since $(b_\mu^2)_\mu$ defines an approximate unit for $\B$ we have
\begin{align*}
\|h\| 
& = \lim_\mu \|\Phi_\tau(b_\mu)\xi_\tau\| 
= \lim_\mu \|b_\mu + N_\tau\| 
= \lim_\mu \tau(b_\mu^2)^{1/2} 
= \|\tau\|^{1/2} = 1,
\end{align*}
and thus $h \neq 0$.
On the other hand we have $\A \subseteq \B$, and so $\A = \A^2 \subseteq [\B \A] \subseteq N_\tau$. 
Hence
\begin{align*}
\|P_\A h\| 
& = \| P_\A \xi_\tau\| 
= \lim_\la \|\Phi_\tau(a_\la) \xi_\tau\| 
= \lim_\la \|a_\la+N_\tau\| 
= \lim_\la \tau(a_\la^2)^{1/2} = 0.
\end{align*}
Therefore $(P_\B - P_\A) h = h \neq 0$, and thus $P_\A \neq P_\B$, as required.

To complete the proof, suppose that $[\A \B] = \B$ and that $\B$ has a unit $1_{\B}$.
It then follows that $\lim_\la a_\la b = b$ for all $b \in \B$, and any approximate unit $(a_\la)_\la$ of $\A$.
In particular we have 
\[
\lim_\la a_\la = \lim_\la a_\la 1_{\B} = 1_{\B},
\]
and as $\A$ is closed we get $1_{\B} \in \A$. 
Conversely, if $\A$ contains the unit of $\B$, then it is evident that $[\A \B] = \B$, as required.
\end{proof}

\section{Hyperrigidity} \label{S:hyper}

\subsection{Hyperrigidity, UEP and maximality}

The ucp hyperrigidity property was introduced by Arveson \cite{Arv11} for a generating set $\G$ in a C*-algebra $\ca(\G)$.
A set $\G$ will be called \emph{ucp hyperrigid} if for every faithful non-degenerate $*$-representation $\Phi \colon \ca(\G) \to \B(H)$ and every sequence of unital completely positive maps $\phi_n \colon \B(H) \to \B(H)$ the following holds:
\[
\lim_n \|\phi_n(\Phi(g)) - \Phi(g)\|=0, \forall g \in \G
\Rightarrow
\lim_n \|\phi_n(\Phi(c)) - \Phi(c)\| = 0, \forall c \in \ca(\G).
\]
In \cite[Theorem 2.1]{Arv11} Arveson provides a reformulation of ucp hyperrigidity in terms of the unique extension property when the generating set is a separable operator system.
As indicated by Kim \cite[page 4]{Kim21}, the proof of \cite[Theorem 2.1]{Arv11} extends to the non-separable case by replacing ``separable'' with ``density character at most $\kappa$'' for any infinite cardinal $\kappa$.

We consider the following variant which appears in \cite[Lemma 2.7]{Sal17}.
We say that $\G$ is \emph{ccp hyperrigid} if for every faithful non-degenerate $*$-representation $\Phi \colon \ca(\G) \to \B(H)$ and every sequence of completely contractive completely positive maps $\phi_n \colon \B(H) \to \B(H)$ the following holds:
\[
\lim_n \|\phi_n(\Phi(g)) - \Phi(g)\|=0, \forall g \in \G
\Rightarrow
\lim_n \|\phi_n(\Phi(c)) - \Phi(c)\| = 0, \forall c \in \ca(\G).
\]
We start by showing that the two notions of hyperrigidity coincide when $\ca(\G)$ is not unital.

\begin{proposition} \label{P:ccp/ucp}
Let $\G\subseteq\ca(\G)$ be a generating set such that $\ca(\G)$ is not unital.
Then $\G$ is ccp hyperrigid if and only if $\G$ is ucp hyperrigid.
\end{proposition}

\begin{proof}
It suffices to show that, if $\G$ is ucp hyperrigid, then it is ccp hyperrigid.
Towards this end, let $\Phi \colon \ca(\G) \to \B(H)$ be a faithful non-degenerate $*$-representation and let $\phi_n \colon \B(H) \to \B(H)$ be completely contractive completely positive maps such that
\[
\lim_n \|\phi_n(\Phi(g)) - \Phi(g)\| = 0 \foral g \in \G.
\]
Since $\ca(\G)$ is not unital, we have that $\Phi(\ca(\G))$ is not unital.
Consider the completely contractive completely positive maps $\phi_n|_{\Phi(\ca(\G))}$ and their unital completely positive extensions 
\[
\left(\phi_n|_{\Phi(\ca(\G))}\right)^{\#} \colon \Phi(\ca(\G))^{\#} := \Phi(\ca(\G)) + \bC I_H \to \B(H),
\]
by \cite[Proposition 2.2.1]{BO08}.
By Arveson's Extension Theorem, for each $n$ let $\psi_n \colon \B(H) \to \B(H)$ be a unital completely positive extension.
Therefore for every $g \in \ca(\G)$ we get
\[
\lim_n \psi_n(\Phi(g)) = \lim_n \phi_n(\Phi(g)) = \Phi(g).
\]
Since $\G$ is ucp hyperrigid we conclude 
\[
\lim_n \phi_n(\Phi(c)) = \lim_n \psi_n(\Phi(c)) = \Phi(c)
\foral
c \in \ca(\G),
\]
as required.
\end{proof}

In Theorem \ref{T:hypiffmaxun} we will show that the two notions of hyperrigidity coincide when $\ca(\G)$ is unital, thus covering all cases. 
For that, we will need to connect ucp hyperrigidity (resp.\ ccp hyperrigidity) with the ucp unique extension property (resp.\ ccp unique extension property).

Let $\Phi \colon \ca(\G) \to \B(H)$ be a $*$-representation. 
We say that $\Phi$ has \emph{the ucp unique extension property with respect to $\G$} if for any faithful non-degenerate $*$-representation $j\colon \ca(\G) \to \B(K)$ and any unital completely positive map $\Psi \colon \B(K) \to \B(H)$ the following holds:
\[
\Psi \circ j(g) = \Phi(g) \foral g \in \G \Rightarrow \Psi \circ j(c) = \Phi(c) \foral c \in \ca(\G).
\]
Following \cite[Definition 2.3]{Sal17}, we say that $\Phi$ has \emph{the ccp unique extension property with respect to $\G$} if for any faithful non-degenerate $*$-representation $j\colon \ca(\G) \to \B(K)$ and any completely contractive completely positive map $\Psi \colon \B(K) \to \B(H)$ the following holds:
\[
\Psi \circ j(g) = \Phi(g) \foral g \in \G \Rightarrow \Psi \circ j(c) = \Phi(c) \foral c \in \ca(\G).
\]
Note that the distinction between the ucp and the ccp unique extension property does not appear when $\G$ contains a unit of $\ca(\G)$, as trivially the ccp extensions must be unital due to non-degeneracy.
Hence this difference does not appear for operator systems.

The following remark is a standard trick.

\begin{remark}\label{R:faithful}
Suppose that every faithful non-degenerate $*$-representation of $\ca(\G)$ has the ucp unique extension property (resp.\ ccp unique extension property) with respect to $\G$.
We claim that every non-degenerate $*$-representation of $\ca(\G)$ has the ucp unique extension property (resp.\ ccp unique extension property) with respect to $\G$. 

Towards this end, suppose that $\Phi \colon \ca(\G) \to \B(H)$ is a non-degenerate $*$-representation, and take $j\colon \ca(\G) \to \B(K)$ be a faithful non-degenerate $*$-representation.
Let $\Psi \colon \B(K) \to \B(H)$ be a unital (resp.\ completely contractive) completely positive map  with $\Psi \circ j|_\G = \Phi|_\G$. 
We will show that $\Psi \circ j = \Phi$.
So let $\Phi' \colon \ca(\G)\to \B(H')$ be a faithful non-degenerate $*$-representation, and consider the faithful non-degenerate $*$-representations 
\[
\wt{\Phi} := \Phi \oplus \Phi'
\qand
\wt{j} := j \oplus \Phi'.
\]
By Arveson's Extension Theorem, let $\wt{\Psi}$ be a unital (resp.\ completely contractive) completely positive extension of $\Psi \oplus \id_{H'}$ on $\B(K \oplus H')$.
We then have 
\[
\wt{\Psi} \circ \wt{j}|_\G = \Phi|_{\G} \oplus \Phi'|_\G = \wt{\Phi}|_\G.
\]
Since $\wt{\Phi}$ has the ucp unique extension property (resp.\ ccp unique extension property) with respect to $\G$, we get $\wt{\Psi} \circ \wt{j} = \wt{\Phi}$, and thus $\Psi \circ j =\Phi$ as required.
\end{remark} 

We have that ucp hyperrigidity (resp.\ ccp hyperrigidity) implies the ucp unique extension property (resp.\ ccp unique extension property) for any non-degenerate $*$-representation. 
This was shown in \cite[Theorem 2.1]{Arv11} when $\G$ is an operator system, and in \cite[Theorem 2.8]{Sal17} when both $\G$ and $\ca(\G)$ are unital, or when both $\G$ and $\ca(\G)$ are non-unital. 
We now give a proof without any assumptions on $\G$ or $\ca(\G)$. 

\begin{proposition}\label{P:hypuep}
Let $\G \subseteq \ca(\G)$ be a generating set.
If $\G$ is ucp hyperrigid (resp.\ ccp hyperrigid), then every non-degenerate $*$-representation of $\ca(\G)$ has the ucp unique extension property (resp.\ ccp unique extension property) with respect to $\G$.
\end{proposition}

\begin{proof}
Let $\Phi \colon \ca(\G)\to \B(H)$ be a non-degenerate $*$-representation and $j\colon \ca(\G) \to \B(K)$ be a faithful non-degenerate $*$-representation.
By Remark \ref{R:faithful} we may suppose that $\Phi$ is faithful. 
Let $\Psi\colon \B(K)\to \B(H)$ be a unital (resp. completely contractive) completely positive map such that $\Psi\circ j|_{\G} = \Phi|_{\G}$. 
We consider two cases.

\smallskip

\noindent{\bf Case 1.}
If $\ca(\G)$ is unital, then $\Phi$ and $j$ are unital, being non-degenerate. 
By Arveson's Extension Theorem, there exists a unital (resp.\ completely contractive) completely positive map $\wt{\Psi}\colon \B(H)\to \B(H)$ extending the unital (resp.\ completely contractive) completely positive map 
\[
\Psi\circ j\circ \Phi^{-1} \colon \Phi(\ca(\G)) \to \B(H).
\]
Set $\phi_n:=\wt{\Psi}$ for every $n$, and note that 
\[
\lim_n \|\phi_n(\Phi(g)) - \Phi(g)\|=\|\Psi\circ j(g)-\Phi(g)\| = 0\foral g\in \G.
\]
Then ucp hyperrigidity (resp.\ ccp hyperrigidity) of $\G$ implies the required
\[
\|\Psi\circ j(c)-\Phi(c)\|=\lim_n \|\phi_n(\Phi(c)) - \Phi(c)\| = 0\foral c\in\ca(\G).
\]

\smallskip

\noindent{\bf Case 2.}
If $\ca(\G)$ is not unital, then $\Phi(\ca(\G))$ is not unital since $\Phi$ is faithful and non-degenerate.
Consider the completely contractive completely positive map 
\[
\Psi\circ j\circ \Phi^{-1} \colon \Phi(\ca(\G)) \to \B(H),
\] 
and take its unital completely positive extension 
\[
(\Psi \circ j \circ \Phi^{-1})^\#\colon \Phi(\ca(\G))^\# \to \B(H)
\]
obtained by \cite[Proposition 2.2.1]{BO08}.
By Arveson's Extension Theorem, let $\wt{\Psi} \colon \B(H) \to \B(H)$ be a unital completely positive map that extends $(\Psi \circ j \circ \Phi^{-1})^\#$ and set $\phi_n:=\wt{\Psi}$ for every $n$. 
Note that for any $c\in\ca(\G)$ we have
\[
\phi_n(\Phi(c))
= (\Psi \circ j \circ \Phi^{-1})^\#(\Phi(c))
= (\Psi \circ j \circ \Phi^{-1})(\Phi(c))
= \Psi\circ j(c),
\]
and hence the rest of the proof works as in the unital case above.
\end{proof}

The connection of the ucp unique extension property (resp.\ ccp unique extension property) with maximality is straightforward and follows standard dilation theory arguments.

\begin{definition} \label{D:max}
Let $\G\subseteq \ca(\G)$ be a generating set and let $\Phi\colon \ca(\G)\to \B(H)$ be a $*$-homomor\-phism. 
We say that $\Phi$ is \emph{maximal on $\G$} if for every $*$-homomorphism $\Phi'\colon \ca(\G)\to \B(H')$ with $H \subseteq H'$ we have that:
\[
\Phi(g) = P_H \Phi'(g) |_H \foral g \in \G
\Longrightarrow
\Phi(c) = \Phi'(c) |_H \foral c \in \ca(\G).
\]
\end{definition}

We should note that if $\Phi$ and $\Phi'$ are as above, then since $\G$ is a generating set, and $\Phi$ and $\Phi'$ are $*$-representations, we get $\Phi'=\Phi\oplus \Phi''$ for a $*$-representation $\Phi''$ of $\ca(\G)$. 

Definition \ref{D:max} is compatible with maximality of maps.
If $\fS$ is an operator system, then by a use of Arveson's Extension Theorem and Stinespring's Dilation Theorem, we obtain that $\Phi$ is maximal on $\fS$ if and only $\Phi|_\fS$ is a maximal unital completely positive map, i.e., it does not admit non-trivial dilations by unital completely positive maps. 
This is also the case for operator algebras. 
In particular, if $\fA$ is a non-unital operator algebra, then by a use of Meyer's Unitisation Theorem and then \cite[Corollary 7.7]{Pau02} we get that $\Phi$ is maximal on $\fA$ if and only $\Phi|_\fA$ is a maximal representation, i.e., it does not admit non-trivial dilations by completely contractive representations of $\fA$.

\begin{lemma} \label{L:iffmaxspace}
Let $\G\subseteq \ca(\G)$ be a generating set and let $\Phi\colon \ca(\G)\to \B(H)$ be a $*$-homomor\-phism. 
The following are equivalent:
\begin{enumerate}
\item $\Phi$ is maximal on $\G$.
\item $\Phi$ is maximal on $\ol{\spn}\{g \mid g \in \G\}$.
\item $\Phi$ is maximal on $\ol{\spn}\{g, g^* \mid g \in \G\}$.
\end{enumerate}
\end{lemma}

\begin{proof}
This is an immediate consequence of the fact that compressions are compatible with linear spans, adjoints and norm limits.
\end{proof}

With an additional property of invariance we get a similar result for the algebra of a generating set.
The following has been implicit in \cite[proof of Claim in Theorem 3.5]{KK12}.

\begin{lemma} \label{L:iffmaxalg}
Let $\G\subseteq \ca(\G)$ be a generating set and let $\Phi\colon \ca(\G)\to \B(H)$ be a $*$-homomor\-phism.
Suppose in addition that for every $*$-homomorphism $\Phi' \colon \ca(\G) \to \B(H')$ satisfying $H \subseteq H'$ and $\Phi(g) = P_H \Phi'(g) |_H$ for all $g \in \G$, we have that $H$ is invariant or co-invariant for $\Phi'(\G)$.
The following are equivalent:
\begin{enumerate}
\item $\Phi$ is maximal on $\G$.
\item $\Phi$ is maximal on $\ol{\alg}\{g \mid g \in \G\}$.
\end{enumerate}
\end{lemma}

\begin{proof}
We will just show the case where $H$ stays invariant.
The co-invariant case follows by applying the dual arguments on $\G^*$.

If item (i) holds, then item (ii) holds trivially.
For the converse, let $\Phi' \colon \ca(\G) \to \B(H')$ be a $*$-representation with $H \subseteq H'$ and $\Phi(g) = P_H \Phi'(g) |_H$ for all $g \in \G$.
By assumption we have
\begin{align*}
\Phi(g_1 g_2) 
& = \Phi(g_1) \Phi(g_2) 
= P_H \Phi'(g_1) P_H \Phi'(g_2)|_H \\
&=
P_H \Phi'(g_1) \Phi'(g_2)|_H 
= P_H \Phi'(g_1 g_2)|_H,
\end{align*}
for all $g_1, g_2 \in \G$.
Hence 
\[
\Phi|_{\ol{\alg}\{\G\}} = P_H (\Phi'|_{\ol{\alg}\{\G\}}) |_H,
\]
and by item (ii) we have $\Phi = \Phi'|_H$, as required.
\end{proof}

We will need the following remark for future reference.

\begin{proposition} \label{P:a-max}
Let $\C$ be a C*-algebra and let $\fY \subseteq \C$ be a subset.
Let $\Phi \colon \C \to \B(H)$ and $\Phi' \colon \C \to \B(H')$ be $*$-representations with $H \subseteq H'$ such that $\Phi|_{\fY}$ is a direct summand of $\Phi'|_{\fY}$.
If for $c \in \C$ there exists $(c_\la)_\la \subseteq \ca(\fY)$ such that 
\[
\textup{wot-}\lim_\la \Phi(c_\la) = \Phi(c)
\qand
\textup{wot-}\lim_\la \Phi'(c_\la) = \Phi'(c),
\]
then $\Phi(c)$ is direct summand of $\Phi'(c)$.
\end{proposition}

\begin{proof}
Since both $\Phi$ and $\Phi'$ are $*$-representations agreeing on the set $\fY$, we have that $\Phi|_{\ca(\fY)}$ is a direct summand of $\Phi'|_{\ca(\fY)}$.
By assumption we have
\[
\Phi'(c)
= \text{wot-}\lim_\la \Phi'(c_\la)
= \text{wot-}\lim_\la \begin{bmatrix} \Phi(c_\la) & 0 \\ 0 & \ast \end{bmatrix}
= \begin{bmatrix} \Phi(c) & 0 \\ 0 & \ast \end{bmatrix}.
\]
This shows that $\Phi(c)$ is a direct summand of $\Phi'(c)$, as required.
\end{proof}

We now show that maximality, the ucp unique extension property and the ccp unique extension property coincide. 
This was shown in \cite[Proposition 2.4]{Arv08} for operator systems and in \cite[Proposition 2.4]{DS18} for operator algebras. Both cases are adaptations of the proof of \cite[Theorem 2]{MS98}. 
With a use of Proposition \ref{P:Dil} the result holds for any generating set.

\begin{proposition}\label{P:uepiffmax}
Let $\G\subseteq \ca(\G)$ be a generating set and $\Phi \colon \ca(\G) \to \B(H)$ be a non-degenerate $*$-representation. 
The following are equivalent:
\begin{enumerate}
\item $\Phi$ has the ccp unique extension property with respect to $\G$.
\item $\Phi$ has the ucp unique extension property with respect to $\G$.
\item $\Phi$ is maximal on $\G$.
\end{enumerate}
\end{proposition}

\begin{proof}
The implication [(i) $\Rightarrow$ (ii)] is immediate.
For the implication [(ii) $\Rightarrow$ (iii)], suppose that $\Phi$ has the ucp unique extension property with respect to $\G$, and consider a $*$-representation $\Phi' \colon \ca(\G)\to B(H')$ such that $H\subseteq H'$ and 
\[
\Phi(g)=P_H\Phi'(g)|_H \foral g\in \G.
\]
Without loss of generality we assume that $\Phi'$ is also non-degenerate.
Let also $j \colon \ca(\G) \to \B(K)$ be a faithful non-degenerate $*$-representation and set $\phi' := P_H \Phi'|_H$.
We consider two cases.

\smallskip

\noindent{\bf Case 1.}
If $\ca(\G)$ is unital, then all maps above are unital.
Consider the unital completely positive map $\Psi_0 := \phi' \circ j^{-1}$.
By Arveson's Extension Theorem, let $\Psi$ be a unital completely positive extension of $\Psi_0$ on $\B(K)$, for which we get 
\[
\Psi \circ j|_\G = \phi'|_{\G} = \Phi|_{\G}.
\]
Since $\Phi$ has the ucp unique extension property with respect to $\G$, we conclude $\phi' = \Psi \circ j = \Phi$, and thus $\phi'$ is multiplicative.
Hence $H$ is reducing for $\Phi'$, i.e., $\Phi$ is a direct summand of $\Phi'$.

\smallskip

\noindent{\bf Case 2.}
If $\ca(\G)$ is not unital, then we have $I_K \notin j(\ca(\G))$, and hence the unitisation map $j^\#\colon \ca(\G)^\#\to \B(K)$ is faithful.
By \cite[Proposition 2.2.1]{BO08} we obtain a unital completely positive extension 
\[
(\phi')^{\#}\colon \ca(\G)^\#\to \B(H)
\]
of the completely contractive completely positive map $\phi':=P_H\Phi'|_H$.
By Arveson's Extension Theorem, let $\Psi$ be a unital completely positive extension of $(\phi')^{\#} \circ (j^{\#})^{-1}$. 
For $g\in \G$ we have
\begin{align*}
\Psi\circ j(g) 
& = (\phi')^\#\circ (j^\#)^{-1}\circ j(g)
= (\phi')^\#\circ j^{-1}\circ j(g) \\
& = (\phi')^\#(g)
= \phi'(g)
= P_H\Phi'(g)|_H
= \Phi(g).
\end{align*}
Since $\Phi$ has the ucp unique extension property with respect to $\G$, we conclude $\phi' = \Psi \circ j = \Phi$, and thus $\Phi$ is a direct summand of $\Phi'$.

\smallskip

It remains to prove the implication [(iii) $\Rightarrow$ (i)].
Towards this end, suppose that $\Phi$ is maximal on $\G$.
Let $j\colon \ca(\G)\to \B(K)$ be a faithful non-degenerate $*$-representation, and let $\Psi\colon \B(K)\to \B(H)$ be a completely contractive completely positive map such that $\Psi\circ j|_{\G}=\Phi|_{\G}$. 
By Proposition \ref{P:Dil} there exist a Hilbert space $H\subseteq H'$ and a $*$-representation $\Phi'\colon \ca(\G)\to \B(H')$ such that $\Psi\circ j = P_H\Phi' |_H$.
We have 
\[
\Phi(g)=\Psi\circ j(g)=P_H \Phi'(g)|_H \foral g\in \G,
\]
and thus maximality of $\Phi$ on $\G$ implies that $H$ is reducing for $\Phi'$.
In particular, we have $\Phi = P_H \Phi' |_H = \Psi\circ j$, and the proof is complete.
\end{proof}

\begin{remark}\label{R:uepiffmax}
We note for future reference that the proof of the implication [(iii) $\Rightarrow$ (i)] does not require $\Phi$ to be non-degenerate.
\end{remark}

\subsection{Unital generated C*-algebra}

In this subsection we prove that ucp hyperrigidity (resp.\ ccp hyperrigidity) is equivalent to any unital $*$-representation having the ucp unique extension property (resp.\ ccp unique extension property) when the generated C*-algebra is unital. 
This is done without any further assumptions on the generating set, and as a result by using maximality we obtain that ccp hyperrigidity and ucp hyperrigidity coincide for any generating set.

In \cite[Theorem 2.1]{Arv11} it is shown that the converse of Proposition \ref{P:hypuep} holds when $\G$ is an operator system, and hence contains a unit for $\ca(\G)$.
This holds also when $\ca(\G)$ is unital without assuming that $\G$ contains the unit.

\begin{theorem}\label{T:hypiffuepun}
Let $\G \subseteq \ca(\G)$ be a generating set such that  $\ca(\G)$ is unital.
The following are equivalent:
\begin{enumerate}
\item $\G$ is ucp hyperrigid (resp.\ ccp hyperrigid).
\item Every non-degenerate $*$-representation $\Phi \colon \ca(\G) \to \B(H)$ has the ucp unique extension property  (resp.\ ccp unique extension property) with respect to $\G$. 
\end{enumerate}
\end{theorem}

\begin{proof}
The implication [(i) $\Rightarrow$ (ii)] follows from Proposition \ref{P:hypuep}, hence it suffices to show the converse.
Let $\Phi \colon \ca(\G) \to \B(H)$ be a faithful non-degenerate (and thus unital) $*$-representation and let $\phi_n \colon \B(H) \to \B(H)$ be unital (resp.\ completely contractive) completely positive maps such that
\[
\lim_n \|\phi_n(\Phi(g)) - \Phi(g)\| = 0 \foral g \in \G.
\]
Consider the inclusion map
\[
\io \colon \B(H) \to \ell^{\infty}(\B(H)); x \mapsto (x, x, \dots)
\]
and the canonical quotient map
\[
q \colon \ell^{\infty}(\B(H)) \to \ell^\infty(\B(H))/c_0(\B(H)).
\]
Let us also fix a faithful unital $*$-representation
\[
\ka \colon \ell^\infty(\B(H))/c_0(\B(H)) \to \B(K).
\]
Define the faithful unital $*$-representation
\[
\wt{\Phi} \colon \ca(\G) \to \B(K); \wt{\Phi}(c) := \ka \circ q \circ \io \circ \Phi(c),
\]
and the unital (resp.\ completely contractive) completely positive map
\[
\wt{\Psi} \colon \wt{\Phi}(\ca(\G)) \to \B(K); \wt{\Psi}_0(\wt{\Phi}(c)) = \ka \circ q ( \phi_1(\Phi(c)), \phi_2(\Phi(c)), \dots ).
\]
By Arveson's Extension Theorem there is a unital (resp.\ completely contractive) completely positive extension of $\wt{\Psi}$ to $\B(K)$ which we denote by the same symbol. 
Since $\wt{\Psi}\circ \wt{\Phi}(g) = \wt{\Phi}(g)$ for any $g\in\G$, by the assumption on the ucp unique extension property (resp.\ ccp unique extension property) for $j=\wt{\Phi}$ we get $\wt{\Psi}\circ \wt{\Phi} = \wt{\Phi}$.
Equivalently, we have
\[
\ka \circ q(\phi_1(\Phi(c)), \phi_2(\Phi(c)), \dots)
=
\ka \circ q(\Phi(c), \Phi(c), \dots)
\foral 
c \in \ca(\G).
\]
Since $\ka$ is faithful we conclude $\lim_n \phi_n(\Phi(c)) = \Phi(c)$ for all $c \in \ca(\G)$, as required.
\end{proof}

The standard way for connecting hyperrigidity with maximality is through the unique extension property. 
This connection provides also the equivalence between the ccp hyperrigidity and the ucp hyperrigidity for a generating set of a unital C*-algebra.

\begin{theorem}\label{T:hypiffmaxun}
Let $\G \subseteq \ca(\G)$ be a generating set such that $\ca(\G)$ is unital.
The following are equivalent:
\begin{enumerate}
\item $\G$ is ccp hyperrigid.
\item $\G$ is ucp hyperrigid.
\item Every non-degenerate $*$-representation $\Phi \colon \ca(\G) \to \B(H)$ is maximal on $\G$.
\end{enumerate}
Moreover we have that $\G$ is ccp/ucp hyperrigid if and only if $\G \cup \{1\}$ is ccp/ucp hyperrigid.
\end{theorem}

\begin{proof}
The first part of equivalences follows by Theorem \ref{T:hypiffuepun} and Proposition \ref{P:uepiffmax}.
For the second part, we have by assumption that $\ca(\G) = \ca(\G \cup \{1\})$.
Hence the conclusion follows by the fact that a unital $*$-representation is maximal on $\G$ if and only if it is maximal on $\G\cup \{1\}$.
\end{proof}

Combining Proposition \ref{P:ccp/ucp} and Theorem \ref{T:hypiffmaxun} we obtain the following corollary.

\begin{corollary}\label{C:ccp/ucp}
Let $\G\subseteq\ca(\G)$ be a generating set.
Then $\G$ is ccp hyperrigid if and only if $\G$ is ucp hyperrigid.
\end{corollary}

\subsection{Separating generating set}

Salomon \cite{Sal17} identifies a significant problem when moving to the non-unital context linked to the rigidity of the zero map.
The set $\G$ is said to be \emph{rigid at zero} if for every sequence $(\vphi_n)_n \subseteq {\rm CCP}(\ca(\G), \bC)$ we have:
\[
\lim_n \vphi_n(g)=0, \forall g \in \G
\Rightarrow
\lim_n \vphi_n(c) = 0, \forall c \in \ca(\G).
\]
We will avoid this terminology as we have excluded the zero map from the non-degenerate $*$-representations.
In \cite[Theorem 3.3]{Sal17} it is shown that a generating set is rigid at zero if and only if it is separating in the following sense.

\begin{definition}
Let $\G \subseteq \ca(\G)$ be a generating set.
We say that $\G$ is \emph{separating for $\ca(\G)$} if for any $\vphi \in {\rm CCP}(\ca(\G), \bC)$ we have that:
\[
\vphi(g) = 0, \forall g \in \G \Rightarrow \vphi(c) = 0, \forall c \in \ca(\G).
\]
\end{definition}

Equivalently $\G$ is a separating set if and only if the states of $\ca(\G)$ do not annihilate $\G$.
Since an approximate unit (of contractive positive elements) in a C*-algebra is separating for its completely contractive completely positive functionals, we have the following class of separating generating sets as proven in \cite[Proposition 3.6]{Sal17}.

\begin{proposition} \label{P:unitcls} \cite[Proposition 3.6]{Sal17}.
Let $\G \subseteq \ca(\G)$ be a generating set.
If the closed linear span of $\G$ contains an approximate unit (of contractive positive elements) for $\ca(\G)$, then $\G$ is separating.
\end{proposition}

We now present the proof of \cite[Theorem 3.3]{Sal17}.
The proof we provide here allows to remove the non-degeneracy assumption from \cite[Theorem 3.3, item (ii)]{Sal17}.

\begin{theorem} \label{T:rigzero} \cite[Theorem 3.3]{Sal17}.
Let $\G \subseteq \ca(\G)$ be a generating set.
The following are equivalent:
\begin{enumerate}
\item $\G$ is a separating set.
\item $\G$ is rigid at zero.
\item For every faithful $*$-representation $\Phi \colon \ca(\G) \to \B(H)$ and every sequence of unital completely positive maps $\phi_n \colon \B(H) \to \B(H)$ the following holds:
\[
\lim_n \phi_n(\Phi(g))=0 \foral g \in \G
\Rightarrow
\lim_n \phi_n(\Phi(c))=0 \foral c \in \ca(\G).
\]
\item For every faithful $*$-representation $\Phi \colon \ca(\G) \to \B(H)$ and every sequence of completely contractive completely positive maps $\phi_n \colon \B(H) \to \B(H)$ the following holds:
\[
\lim_n \phi_n(\Phi(g))=0 \foral g \in \G
\Rightarrow
\lim_n \phi_n(\Phi(c))=0 \foral c \in \ca(\G).
\]
\end{enumerate}
\end{theorem}

\begin{proof}

\noindent
[(i) $\Leftrightarrow$ (ii)].
Suppose that $\G$ is a separating set and let $(\vphi_n)_n \subseteq {\rm CCP}(\ca(\G), \bC)$ with $\lim_n \vphi_n(g)=0$ for every $g \in \G$.
Consider the completely contractive completely positive map
\[
\Psi \colon \ca(\G) \to \ell^\infty; \Psi(c) = (\vphi_1(c), \vphi_2(c), \dots),
\]
and let $q \colon \ell^\infty \to \ell^\infty/c_0$ denote the canonical quotient map.
For every state $\tau$ of $\ell^\infty/c_0$ we have that the map
\[
\vphi_\tau := \tau \circ q \circ \Psi
\]
is a completely contractive completely positive map of $\ca(\G)$ such that
\[
\vphi_\tau|_{\G} = \tau \circ q \circ \Psi|_{\G} = 0,
\]
since $\Psi(\G) \subseteq \ker q$ by assumption.
Since $\G$ is separating, we then get $\vphi_\tau = 0$.
Since this holds for every state $\tau$ of $\ell^\infty/c_0$ we deduce that $q \circ \Psi(c) = 0$ for every positive $c \in \ca(\G)$, and hence $q\circ \Psi=0$.
Therefore we conclude $\Psi(\ca(\G)) \subseteq c_0$; that is $\lim_n \vphi_n(c) = 0$ for every $c \in \ca(\G)$.

Conversely, suppose that item (ii) holds, and suppose that for some $\vphi \in {\rm CCP}(\ca(\G), \bC)$ we have $\vphi|_{\G} = 0$.
We can then set $\vphi_n = \vphi$ for every $n \in \bN$, and deduce that $\vphi = 0$.

\smallskip

\noindent
[(i) $\Leftrightarrow$ (iii)].
Suppose that item (i) holds.
Let $\Phi \colon \ca(\G) \to \B(H)$ be a faithful $*$-representation and let $\phi_n \colon \B(H) \to \B(H)$ be unital completely positive maps such that 
\[
\lim_n \phi_n(\Phi(g)) = 0 \foral g \in \G.
\]
Consider the canonical quotient map
\[
q \colon \ell^{\infty}(\B(H)) \to \ell^\infty(\B(H))/c_0(\B(H)).
\]
Then for any state $\tau$ in $\ell^\infty(\B(H))/c_0(\B(H))$ we can consider the completely contractive completely positive functional
\[
\vphi_{\tau}(c) := \tau \circ q(\phi_1(\Phi(c)), \phi_2(\Phi(c)), \dots)
\foral
c \in \ca(\G).
\]
By definition we have $\vphi_{\tau}|_{\G} = 0$, and thus the separating property gives $\vphi_{\tau}=0$.
Since states of C*-algebras separate their points, we get
\[
q(\phi_1(\Phi(c)), \phi_2(\Phi(c)), \dots) = 0
\foral
c \in \ca(\G),
\]
and thus $\lim_n \phi_n(\Phi(c)) = 0$ for all $c \in \ca(\G)$.

The converse is proven in the same way as in \cite[Theorem 3.3, (ii) implies (v)]{Sal17}.
Suppose that for some $\vphi \in  {\rm CCP}(\ca(\G), \bC)$ we have $\vphi|_{\G} = 0$.
Let $\Phi \colon \ca(\G) \to \B(H)$ be a faithful $*$-representation, and take the completely contractive completely positive map 
\[
\vphi' := \vphi \circ \Phi^{-1} \colon \Phi(\ca(\G)) \to \bC.
\]
Then we can extend $\vphi'$ to a completely contractive completely positive map on $\B(H)$, which we will denote by the same symbol $\vphi'$. 
After normalising, without loss of generality we may assume that $\vphi'(I_H) = \|\vphi'\| = 1$.
Let $\phi \colon \B(H) \to \B(H)$ be the unital completely positive map given by $\phi(x) := \vphi'(x) I_H$.
Then 
\[
\phi(\Phi(g)) = \vphi'(\Phi(g))I_H = 0
\foral g \in \G.
\]
By considering the constant sequence $\phi_n := \phi$ for all $n \in \bN$, we get $\phi(\Phi(c)) = 0$, and thus $\vphi = 0$.

\smallskip

\noindent
[(i) $\Leftrightarrow$ (iv)].
Follows in the same way as the proof of [(i) $\Leftrightarrow$ (iii)].
\end{proof}
We will need the following lemma to connect hyperrigidity with the unique extension property in the case where the generating set is separating.

\begin{lemma}\label{L:sepdecomp}
Let $\G\subseteq \ca(\G)$ be a separating generating set and consider $\Phi \colon \ca(\G) \to \B(H)$ to be a $*$-representation. 
Set $L:= [\Phi(\ca(\G)) H]$. If $j\colon \ca(\G)\to \B(K)$ is a $*$-representation and $\Psi\colon \B(K)\to \B(H)$ is a completely contractive completely positive map such that $\Psi(j(g))=\Phi(g)$ for every $g\in \G$, then 
\[
\Psi(j(c))=P_L \Psi(j(c))|_L\oplus 0_{L^\perp} \foral c\in \ca(\G).
\]
\end{lemma}
\begin{proof}
Set $\wt{\Psi}:=P_L \Psi|_L$ and we will prove that
\[
\Psi \circ j 
=
(\wt{\Psi} \circ j) \oplus 0_{L^{\perp}}.
\]
Towards this end, let $h \in L^\perp$, and note that
\begin{align*}
\sca{\Psi(j(g))h,h}
& =
\sca{\Phi(g)h,h}
=
0
\foral g \in \G,
\end{align*}
by the definition of $L$.
Since $\G$ is separating we get
\[
\sca{\Psi(j(c))h,h}
=
0
\foral c \in \ca(\G),
\]
and so $P_{L^\perp} (\Psi \circ j) |_{L^\perp} = 0$. 
Therefore, for any positive element $c\in\ca(\G)$ we obtain that the positive matrix $\Psi \circ j(c)$ has a zero $(2,2)$-entry with respect to the decomposition $H=L\oplus L^\perp$. Thus it also has zero $(1,2)$ and $(2,1)$-entries.
Since $\ca(\G)$ is spanned by its positive elements we obtain that $\Psi \circ j $ is zero outside the compression on $L$, as required.
\end{proof}

It turns out that, if $\G\subseteq \ca(\G)$ is separating, then the non-degeneracy condition in the definition of hyperrigidity is redundant. 
This is shown in \cite[Theorem 3.9]{Sal17}; however in the context of \cite[Theorem 3.9]{Sal17} the set $\G$ a priori contains the unit when $\ca(\G)$ is unital.
Since here we do not make this assumption we revisit the proof.

\begin{proposition}\label{P:sephyp} \cite[Theorem 3.9]{Sal17}
Let $\G\subseteq \ca(\G)$ be a separating generating set. 
The following are equivalent:
\begin{enumerate}
\item $\G$ is ccp hyperrigid.
\item $\G$ is ucp hyperrigid.
\item For every faithful $*$-representation $\Phi \colon \ca(\G) \to \B(H)$ and every sequence of completely contractive completely positive maps $\phi_n \colon \B(H) \to \B(H)$ the following holds:
\[
\lim_n \|\phi_n(\Phi(g)) - \Phi(g)\|=0, \forall g \in \G
\Rightarrow
\lim_n \|\phi_n(\Phi(c)) - \Phi(c)\| = 0, \forall c \in \ca(\G).
\]
\item For every faithful $*$-representation $\Phi \colon \ca(\G) \to \B(H)$ and every sequence of unital completely positive maps $\phi_n \colon \B(H) \to \B(H)$ the following holds:
\[
\lim_n \|\phi_n(\Phi(g)) - \Phi(g)\|=0, \forall g \in \G
\Rightarrow
\lim_n \|\phi_n(\Phi(c)) - \Phi(c)\| = 0, \forall c \in \ca(\G).
\]
\end{enumerate}
\end{proposition}

\begin{proof}
The implications [(iii) $\Rightarrow$ (iv) $\Rightarrow$ (ii)] are immediate, while the equivalence [(i) $\Leftrightarrow$ (ii)] is Corollary \ref{C:ccp/ucp}.
It remains to prove the implication [(i) $\Rightarrow$ (iii)]. 
Towards this end, let $\Phi \colon \ca(\G) \to \B(H)$ be a faithful $*$-representation and let $\phi_n \colon \B(H) \to \B(H)$ be completely contractive completely positive maps such that
\[
\lim_n \|\phi_n(\Phi(g)) - \Phi(g)\| = 0 \foral g \in \G.
\]
Consider the inclusion map
\[
\io \colon \B(H) \to \ell^{\infty}(\B(H)); x \mapsto (x, x, \dots)
\]
and the canonical quotient map
\[
q \colon \ell^{\infty}(\B(H)) \to \ell^\infty(\B(H))/c_0(\B(H)).
\]
Let us also fix a faithful unital $*$-representation
\[
\ka \colon \ell^\infty(\B(H))/c_0(\B(H)) \to \B(K).
\]
Define the faithful $*$-representation
\[
\wt{\Phi}_0 \colon \ca(\G) \to \B(K); \wt{\Phi}_0(c) := \ka \circ q \circ \io \circ \Phi(c),
\]
and set
\[
\wt{\Phi} := \wt{\Phi}_0 |_{L}
\, \text{ for } \,
L := [\wt{\Phi}_0(\ca(\G)) K].
\]
It follows that $\wt{\Phi}$ is a faithful non-degenerate $*$-representation of $\ca(\G)$.
Consider the completely contractive completely positive map
\[
\wt{\Psi}_0 \colon \wt{\Phi}(\ca(\G)) \to \B(K); \wt{\Psi}_0(\wt{\Phi}(c)) = \ka \circ q ( \phi_1(\Phi(c)), \phi_2(\Phi(c)), \dots ).
\]
By Arveson's Extension Theorem we obtain a completely contractive completely positive extension of $\wt{\Psi}_0$ on $\B(L)$, which we still denote by $\wt{\Psi}_0$.
Consider the completely contractive completely positive map
\[
\wt{\Psi} := P_{L} \wt{\Psi}_0 |_{L} \colon \B(L) \to \B(L).
\]
By definition we have $\wt{\Phi}_0 = \wt{\Phi} \oplus 0_{L^\perp}$ and $\wt{\Psi}_0 \circ \wt{\Phi}(g) = \wt{\Phi}_0(g)$ for all $g \in \G$.
Therefore by Lemma \ref{L:sepdecomp} we obtain
\[
\wt{\Psi}_0 \circ \wt{\Phi} 
=
(\wt{\Psi} \circ \wt{\Phi}) \oplus 0_{L^{\perp}}.
\]
Now to complete the proof, for every $g \in \G$ we have 
\[
\wt{\Psi} \circ \wt{\Phi}(g)
=
P_L \wt{\Psi}_0 \circ \wt{\Phi}(g) |_{L}
=
P_L \wt{\Phi}_0(g)|_{L}
=
\wt{\Phi}(g).
\]
Since $\G$ is ccp hyperrigid, by Proposition \ref{P:hypuep} we obtain that the faithful non-degenerate $*$-representation $\wt{\Phi}$ has the ccp unique extension property and hence we get $\wt{\Psi}\circ \wt{\Phi} = \wt{\Phi}$.
Since we have annihilation outside $L$, we then get
\[
\wt{\Psi}_0 \circ \wt{\Phi}
=
(\wt{\Psi} \circ \wt{\Phi}) \oplus 0_{L^\perp}
=
\wt{\Phi} \oplus 0_{L^\perp}
=
\wt{\Phi}_0.
\]
Equivalently, we have
\[
\ka \circ q(\phi_1(\Phi(c)), \phi_2(\Phi(c)), \dots)
=
\ka \circ q(\Phi(c), \Phi(c), \dots)
\foral 
c \in \ca(\G).
\]
Since $\ka$ is faithful, we conclude $\lim_n \phi_n(\Phi(c)) = \Phi(c)$ for all $c \in \ca(\G)$, as required.
\end{proof}

\begin{remark} \label{R:Salun}
One of the assumptions in \cite{Sal17} is to add the unit to the generating set of a unital C*-algebra.
In view of \cite[Proposition 3.6]{Sal17}, recorded as Proposition \ref{P:unitcls} here, this addition makes a generating set of a unital C*-algebra de facto separating.
However, there are non-separating ucp hyperrigid non-unital sets generating unital C*-algebras.
Such an example is exhibited in \cite[Example 3.4]{Sal17} as below.

Consider $\{I, S_1, \dots, S_d\}$ in the Cuntz algebra $\O_d$.
Then any unital $*$-representation of $\O_d$ is maximal on $\{I, S_1, \dots, S_d\} \subseteq \O_d$, see \cite{Dun10}.
Adding or removing the unit from the generating set does not change this, and thus any unital $*$-representation of $\O_d$ is maximal on $\{S_1, \dots, S_d\}$.
Hence $\{S_1, \dots, S_d\} \subseteq \O_d$ is ucp hyperrigid by Theorem \ref{T:hypiffmaxun}.
On the other hand, the zero map on  $\{S_1, \dots, S_d\}$ admits different extensions on $\O_d$, e.g., take the map $\phi = S_\mu^* \cdot S_\mu$ for the word $\mu = 1 2 \dots d$ when $d \geq 2$, and the integral over $\bT$ when $d=1$.
Thus $\{S_1, \dots, S_d\} \subseteq \O_d$ is not separating. 
Example \ref{E:notunit hyp not rig} that follows below gives further examples of this phenomenon.
\end{remark}

The implication [(iii) $\Rightarrow$ (i)] of the following theorem appears in \cite[Lemma 2.7, item (ii)]{Sal17} without the separating assumption; in Remark \ref{R:Salunit} we comment why we revisit the proof.

\begin{theorem}\label{T:hypiffuep}
Let $\G \subseteq \ca(\G)$ be a separating generating set.
The following are equivalent:
\begin{enumerate}
\item $\G$ is ccp hyperrigid.
\item $\G$ is ucp hyperrigid.
\item Every non-degenerate $*$-representation $\Phi \colon \ca(\G) \to \B(H)$ has the ccp unique extension property with respect to $\G$.
\item Every non-degenerate $*$-representation $\Phi \colon \ca(\G) \to \B(H)$ has the ucp unique extension property with respect to $\G$.
\item Every non-degenerate $*$-representation $\Phi \colon \ca(\G) \to \B(H)$ is maximal on $\G$.
\item Every $*$-representation $\Phi \colon \ca(\G) \to \B(H)$ is maximal on $\G$.
\item Every $*$-representation $\Phi \colon \ca(\G) \to \B(H)$ has the ccp unique extension property with respect to $\G$.
\item Every $*$-representation $\Phi \colon \ca(\G) \to \B(H)$ has the ucp unique extension property with respect to $\G$.
\end{enumerate}
\end{theorem}

\begin{proof}
The equivalence [(i) $\Leftrightarrow$ (ii)] is Corollary \ref{C:ccp/ucp}, and the implication [(i) $\Rightarrow$ (iii)] is Proposition \ref{P:hypuep}. The equivalence [(iii) $\Leftrightarrow$ (iv) $\Leftrightarrow$ (v)] is Proposition \ref{P:uepiffmax}.
The implication [(vi) $\Rightarrow$ (vii)] is Remark \ref{R:uepiffmax} and the implications [(vii) $\Rightarrow$ (viii) $\Rightarrow$ (iii)] are immediate.
We will complete the proof by showing [(v) $\Rightarrow$ (vi)] and [(vi) $\Rightarrow$ (ii)].

We first prove the implication [(v) $\Rightarrow$ (vi)]. Let $\Phi\colon \ca(\G) \to \B(H)$ and $\Phi'\colon \ca(\G)\to\B(H')$ be $*$-representations such that $H'\supseteq H$ and $\Phi(g)=P_H\Phi'(g)|_H$ for every $g\in\G$.
Moreover, set $L:=[\Phi(\ca(\G))H]$ and $\Phi_0:=\Phi|_L$.
Since $\Phi_0$ is maximal and
\[
P_L \Phi'(g)|_L = P_L \Phi(g)|_L = \Phi_0(g) \foral g \in \G,
\]
we have
\[
\Phi' = \Phi_0 \oplus P_{H' \ominus L} \Phi' |_{H' \ominus L}.
\]
Since
\[
P_{H \ominus L} \Phi'(g) |_{H \ominus L}
=
P_{H \ominus L} \Phi(g) |_{H \ominus L}
=
0
\foral g \in \G,
\]
by the separating property we have $P_{H \ominus L} \Phi' |_{H \ominus L} = 0$.
By positivity of the map $P_{H' \ominus L} \Phi' |_{H' \ominus L}$ and since $\ca(\G)$ is linearly spanned by positive elements we have
\[
P_{H' \ominus L} \Phi' |_{H' \ominus L} = 0_{H \ominus L} \oplus P_{H' \ominus H} \Phi' |_{H' \ominus H}.
\]
Therefore we obtain
\[
\Phi'
= \Phi_0 \oplus P_{H' \ominus L} \Phi' |_{H' \ominus L}
= \Phi_0 \oplus 0_{H \ominus L} \oplus P_{H' \ominus H} \Phi' |_{H' \ominus H}
= \Phi \oplus P_{H' \ominus H} \Phi' |_{H' \ominus H},
\]
i.e., $\Phi$ is a direct summand of $\Phi'$, as required.

We now prove the implication [(vi) $\Rightarrow$ (ii)]. By Theorem \ref{T:hypiffmaxun} it suffices to consider $\ca(\G)$ to be non-unital.
Let $\Phi \colon \ca(\G) \to \B(H)$ be a faithful non-degenerate $*$-representation and let $\phi_n \colon \B(H) \to \B(H)$ be unital completely positive maps such that
\[
\lim_n \|\phi_n(\Phi(g)) - \Phi(g)\| = 0 \foral g \in \G.
\]
Consider the inclusion map
\[
\io \colon \B(H) \to \ell^{\infty}(\B(H)); x \mapsto (x, x, \dots)
\]
and the canonical quotient map
\[
q \colon \ell^{\infty}(\B(H)) \to \ell^\infty(\B(H))/c_0(\B(H)).
\]
Let us also fix a faithful unital $*$-representation
\[
\ka \colon \ell^\infty(\B(H))/c_0(\B(H)) \to \B(K).
\]
Define the faithful $*$-representation
\[
\wt{\Phi} \colon \ca(\G) \to \B(K); \wt{\Phi}(c) := \ka \circ q \circ \io \circ \Phi(c),
\]
and the completely contractive completely positive map
\[
\wt{\Psi} \colon \ca(\G) \to \B(K); \wt{\Psi}(c) = \ka \circ q ( \phi_1(\Phi(c)), \phi_2(\Phi(c)), \dots ).
\]
Proposition \ref{P:Dil} implies that there exists a $*$-representation $\wt{\Phi}'\colon \ca(\G)\to \B(K')$ such that $K'\supseteq K$ and
\[
\wt{\Psi}(c)=P_K \wt{\Phi}'(c)|_K \foral c\in \ca(\G).
\]
For every $g \in \G$ we have
\[
P_K \wt{\Phi}'(g)|_K
=
\wt{\Psi}(g)
=
\ka \circ q ( \phi_1(\Phi(g)), \phi_2(\Phi(g)), \dots )
=
\ka \circ q \circ \io \circ \Phi(g)
=
\wt{\Phi}(g).
\]
By our assumption $\wt{\Phi}$ is maximal on $\G$, and therefore we obtain that $K$ is reducing for $\wt{\Phi}'$.
In particular we have $\wt{\Psi}(c)=\wt{\Phi}(c)$ for every $c\in\ca(\G)$.
Equivalently, we have
\[
\ka \circ q(\phi_1(\Phi(c)), \phi_2(\Phi(c)), \dots)
=
\ka \circ q(\Phi(c), \Phi(c), \dots)
\foral
c \in \ca(\G).
\]
Since $\ka$ is faithful we conclude that the equality $\lim_n \phi_n(\Phi(c)) = \Phi(c)$ holds for all $c \in \ca(\G)$, as required.
\end{proof}

The following appears in \cite[Proposition 3.8, Theorem 3.9]{Sal17} showing in action how the unitisation affects the hyperrigidity of a set.
We include a proof for completeness.

\begin{proposition} \label{P:unit ccp/ucp} \cite[Proposition 3.8, Theorem 3.9]{Sal17}.
Let $\G \subseteq \ca(\G)$ be a generating set such that $\ca(\G)$ is not unital.
Let $\ca(\G)^{\#}$ be the unitisation of $\ca(\G)$ and let $1^{\#}$ be its unit.
The following are equivalent:
\begin{enumerate}
\item $\G \cup \{1^{\#}\}$ is ccp/ucp hyperrigid in $\ca(\G)^\#$.
\item $\G$ is separating and ccp/ucp hyperrigid in $\ca(\G)$.
\end{enumerate}
\end{proposition}

\begin{proof}
Assume that item (i) holds and we will show that item (ii) holds.
Towards this end, let $\Phi \colon \ca(\G) \to \B(H)$ be a faithful and non-degenerate $*$-representation and let $\phi_n \colon \B(H) \to \B(H)$ be unital completely positive maps such that
\[
\lim_n \phi_n(\Phi(g)) = \Phi(g) \foral g \in \G.
\]
Let $\Phi^{\#} \colon \ca(\G)^{\#} \to \B(H)$ be the unitisation map such that $\Phi^{\#}(c + \la 1^{\#}) = \Phi(c) + \la I_H$.
Since $\ca(\G)$, and thus $\Phi(\ca(\G))$, is not unital we have that $\Phi^{\#}$ is also faithful, and by definition it is unital satisfying
\[
\lim_n \phi_n(\Phi^{\#}(g)) = \Phi^{\#}(g) \foral g \in \G \cup \{1^{\#}\}.
\]
Since $\G \cup \{1^{\#}\}$ is ucp hyperrigid we get
\[
\lim_n \phi_n(\Phi^{\#}(c + \la 1^{\#})) = \Phi^{\#}(c + \la 1^{\#}) \foral c + \la 1^{\#} \in \ca(\G \cup \{1^{\#}\}) = \ca(\G)^{\#}.
\]
Applying for $\la=0$ gives
\[
\lim_n \phi_n(\Phi(c)) = \Phi(c) \foral c \in \ca(\G),
\]
and thus $\G$ is ucp hyperrigid.

For the separating property, let $\vphi \colon \ca(\G) \to \bC$ be a non-zero completely contractive functional such that $\vphi|_{\G} = 0$.
By normalising we may assume that $\vphi$ is a state on $\ca(\G)$, and thus it has a unique extension on $\ca(\G)^{\#}$.
Consider the faithful unital $*$-representation
\[
\Phi' \colon \ca(\G)^{\#} \to \B(H \oplus \bC); \Phi'(c + \la 1^{\#}) = \Phi^\#(c+\la 1^{\#}) \oplus \la,
\]
and the unital completely positive map
\[
\Psi \colon \Phi'(\ca(\G)^{\#}) \to \B(H \oplus \bC); \Psi \circ \Phi'(c + \la 1^{\#}) = \Phi^\#(c+ \la 1^{\#}) \oplus \vphi(c + \la 1^{\#}).
\]
We may also consider a unital completely positive extension of $\Psi$ on $\B(H \oplus \bC)$ provided by Arveson's Extension Theorem, which we denote by the same symbol.
For all $g \in \G$ we have
\[
\Psi \circ \Phi'(g) = \Phi(g) \oplus \vphi(g) = \Phi(g) \oplus 0 = \Phi'(g),
\]
and moreover we have
\[
\Psi \circ \Phi'(1^{\#}) = 1_H \oplus 1 = \Phi'(1^{\#}).
\]
Since $\G \cup \{1^{\#}\}$ is ucp hyperrigid, then $\Phi'$ has the ucp unique extension property by Proposition \ref{P:hypuep}, and thus $\Psi \circ \Phi' = \Phi'$.
In particular for $c \in \ca(\G)$ we have
\[
\Phi(c) \oplus \vphi(c) = \Psi \circ \Phi'(c) = \Phi'(c) = \Phi(c) \oplus 0,
\]
and thus $\vphi(c) = 0$ for all $c \in \ca(\G)$, which is a contradiction.

Conversely, assume that item (ii) holds, and we will show that item (i) holds.
Since $\ca(\G)^{\#}$ is unital, by Theorem \ref{T:hypiffmaxun} it suffices to show that every unital $*$-representation $\Phi\colon \ca(\G)^{\#} \to \B(H)$ is maximal on $\G\cup\{1^\#\}$.
Towards this end, let $\Phi' \colon \ca(\G)^{\#} \to \B(H')$ be a unital $*$-representation such that $H'\supseteq H$ and 
\[
\Phi(g)=P_H\Phi'(g)|_H \foral g\in G\cup\{1^\#\}.
\]
Consider the $*$-representations of $\ca(\G)$ given by
\[
\Phi_0:= \Phi|_{\ca(\G)}
\qand
\Phi'_0 := \Phi'|_{\ca(\G)};
\]
in particular we see that for every $g \in \G$ we have
\[
\Phi_0(g)
=
P_H \Phi'_0(g)|_H.
\]
The generating set $\G\subseteq \ca(\G)$ is assumed to be separating and ucp hyperrigid, therefore by Theorem \ref{T:hypiffuep} we obtain that $\Phi_0$ is maximal on $\G$. 
Thus $H$ is reducing for $\Phi'_0(\ca(\G))$, and hence so is for $\Phi'(\ca(\G)^\#)$ since $\Phi'$ is unital.
\end{proof}

\begin{remark} \label{R:unit ccp/ucp}
Proposition \ref{P:hypuep} gives an alternative way for showing that, if $\G$ is not separating and $\ca(\G)$ is not unital, then $\G \cup \{1^{\#}\}$ is not ucp hyperrigid.
Under these assumptions, it can be shown that the one-dimensional representation 
\[
\Phi \colon \ca(\G)^{\#} \to \bC; c + \la 1^{\#} \mapsto \la,
\]
does not have the ucp unique extension property.

Towards this end, since $\G$ is not separating there exists a non-zero state $\vphi \colon \ca(\G) \to \bC$ with $\vphi|_{\G} = 0$.
Denote by the same symbol its extension on the unitisation $\ca(\G)^{\#}$ (recalling that $\ca(\G)$ is not unital).
For a faithful unital $*$-representation $j \colon \ca(\G)^{\#} \to \B(K)$, let $\Psi \colon j(\ca(\G)^{\#}) \to \bC$ be the state given by $\Psi \circ j = \vphi$.
Denote by the same symbol its extension on $\B(K)$.
We then have
\[
\Psi \circ j(1^{\#}) = \vphi(1^{\#}) = 1 = \Phi(1^{\#})
\qand
\Psi \circ j(g) = \vphi(g) = 0 = \Phi(g),
\forall g \in \G.
\]
However for any $c\in\ca(\G)$ such that $c \notin \G$ with $\vphi(c) \neq 0$ we have
\[
\Psi \circ j(c) = \vphi(c) \neq 0 = \Phi(c),
\]
and thus $\Phi$ has not the ucp unique extension property.

The results of \cite[Section 4 and Section 5]{Sal17} provide an abundance of generating sets $\G$ with $\ca(\G)$ not unital, that are ucp hyperrigid and not separating, and thus $\G \cup \{1^{\#}\}$ is not ucp hyperrigid.
Such examples are given by the sets $\{S_e\}_{e \in G^{(1)}}$ of the partial isometries of the edges in a connected row-finite graph with infinite vertices, for example for the graph
\[
\xymatrix{ v_1 & v_2 \ar[l]_{e_1} & v_3 \ar[l]_{e_2} & \ar[l]_{e_3} \cdots} .
\]
In Example \ref{E:notunit hyp not rig} we will work out this example from first principles.
\end{remark}

In \cite[Lemma 2.7 item (ii)]{Sal17} it is asserted that the implication [(iii) $\Rightarrow$ (i)] of Theorem \ref{T:hypiffuep} holds without the assumption of $\G$ being a separating set.
However the suggested strategy in the comments prior to \cite[Lemma 2.7]{Sal17} seems to apply only when $\ca(\G)$ is unital as we explain in the next remark.

\begin{remark} \label{R:Salunit}
The strategy for the proof of \cite[Lemma 2.7 item (ii)]{Sal17} uses the argument that, if $\Phi \colon \ca(\G) \to \B(H)$ is a faithful non-degenerate $*$-representation, then for the $*$-homomorphism 
\[
q\circ\iota \colon \B(H) \to \ell^\infty(\B(H))/c_0(\B(H)) ; x \mapsto (x, x, \dots) + c_0(\B(H)),
\]
and \emph{any} unital $*$-representation $\ka \colon \ell^\infty(\B(H))/c_0(\B(H)) \to \B(K)$, the map $\ka \circ q\circ \io \circ \Phi$ is non-degenerate.
However this is not correct when $\ca(\G)$ is not unital.

If this were true, then for any unital $*$-representation 
\[
\ka \colon \ell^\infty(\B(H))/c_0(\B(H)) \to \B(K)
\]
we would have
\[
[\ka(q\circ \io \circ \Phi(\ca(\G))) K] = K = [\ka(\ell^\infty(\B(H))/c_0(\B(H))) K],
\]
as $\ka$ is unital.
Applying Proposition \ref{P:multiplier} for 
\[
\A = q\circ \io \circ \Phi(\ca(\G))
\qand 
\B=\C=\ell^\infty(\B(H))/c_0(\B(H))
\]
would give that $q\circ \io \circ \Phi(\ca(\G))$ (and thus $\ca(\G)$) is unital.

Nevertheless, we note that a proof of \cite[Lemma 2.7 item (ii)]{Sal17} would still be valid with the existence of just \emph{one} unital faithful $*$-representation $\ka \colon \ell^\infty(\B(H))/c_0(\B(H)) \to \B(K)$ rendering $\ka \circ q\circ \io \circ \Phi$ non-degenerate.
With that in hand one can show the equivalence [(i) $\Leftrightarrow$ (iii)] of Theorem \ref{T:hypiffuep} for any $\G$.
\end{remark}

Below we provide an example of a non-separating ucp hyperrigid set $\G$ for which $\ca(\G)$ is not unital, and hence Proposition \ref{P:unit ccp/ucp} yields that $\G \cup \{1^{\#}\}$ is not ucp hyperrigid in $\ca(\G)^\#$.
We do this in first principles, but the arguments essentially follow the strategies of \cite{Sal17}.

\begin{example} \label{E:notunit hyp not rig}
For notational convenience we write $[m]:=\{1, \dots, m\}$ for $m \in \bN$.
Consider the graph $G$ given by
\[
\xymatrix{ v_1 & v_2 \ar[l]_{e_1} & v_3 \ar[l]_{e_2} & \ar[l]_{e_3} \cdots}
\]
and we write $G_m$ for the finite connected subgraph from $v_1$ to $v_m$.
Let 
\[
\ca(G) = \ca(P_{v_m}, S_{e_m} \mid m \in \bN)
\]
be the C*-algebra generated by a Cuntz-Krieger family $\{P_{v_m}, S_{e_m} \mid m \in \bN\}$ of projections and partial isometries.
By construction, for each $m \geq 2$ we have that the family 
\[
\{P_{v_j}, S_{e_k} \mid j \in [m], k \in [m-1]\}
\]
is a Cuntz-Krieger family for $G_m$, which inherits a gauge action from $\ca(G)$.
Hence an application of the Gauge-Invariant Uniqueness Theorem for graphs \cite{HR97} implies
\[
\ca(G_m) \simeq \ca(\{P_{v_j}, S_{e_k} \mid j \in [m], k \in [m-1]\}).
\]
Thus $\ca(G_m)$ is a C*-subalgebra of $\ca(G)$ for all $m \geq 2$.
Consider the sets
\[
\G := \{S_{e_m} \mid m \in \bN\}
\qand
\G_m := \{S_{e_k} \mid  k \in [m-1]\}, m \in \bN.
\]
By construction we have
\[
\ca(\G) = \ca(G)
\qand
\ca(\G_m) = \ca(G_m) \foral m \in \bN.
\]

First we claim that each $\G_m$ is ucp hyperrigid inside $\ca(G_m)$, see also \cite[Theorem 4.1]{Sal17} of which proof we follow here.
Note here that each $\ca(G_m)$ is unital with unit given by $\sum_{k=1}^m P_{v_k}$.
Hence by Theorem \ref{T:hypiffmaxun} it suffices to show that every unital $*$-representation $\Phi \colon \ca(G_m) \to \B(H)$ is maximal on $\G_m$.
Let $\Phi' \colon \ca(G_m) \to \B(H')$ be a $*$-representation satisfying $H \subseteq H'$ and $\Phi|_{\G_m} = P_H (\Phi'|_{\G_m}) |_{H}$, and without loss of generality assume that $\Phi'$ is also unital.
Then we have that the family
\[
\{\Phi'(P_{v_j}), \Phi'(S_{e_k}) \mid j \in [m], k \in [m-1]\}
\]
is also a Cuntz-Krieger family, and we write
\[
\Phi'(S_{e_k}) = \begin{bmatrix} \Phi(S_{e_k}) & x_k \\ y_k & z_k \end{bmatrix}
\foral k \in [m-1].
\]
We can then write the unit in two ways; as in
\[
I_{H'} = \sum_{k=1}^m \Phi'(P_{v_k}) = (\sum_{k=1}^{m-1} \Phi'(S_{e_k}) \Phi'(S_{e_k})^*) + \Phi'(S_{e_{m-1}})^* \Phi'(S_{e_{m-1}}),
\]
and as in
\[
I_{H'} = \sum_{k=1}^m \Phi'(P_{v_k}) = \Phi'(S_{e_1}) \Phi'(S_{e_1})^* + (\sum_{k=1}^{m-1} \Phi'(S_{e_k})^* \Phi'(S_{e_k})).
\]
The same applies for $\Phi$ in place of $\Phi'$ and for $I_H$ in place of $I_{H'}$.
Therefore by equating the $(1,1)$-entries in each writing, we obtain
\begin{align*}
I_{H} 
& = 
\sum_{k=1}^{m-1} (\Phi(S_{e_k}) \Phi(S_{e_k})^* + x_k x_k^*) + \Phi(S_{e_{m-1}})^* \Phi(S_{e_{m-1}}) + y_{m-1}^* y_{m-1} \\
& = 
\sum_{k=1}^m \Phi(P_{v_k}) + \sum_{k=1}^{m-1} x_k x_k^* + y_{m-1}^*y_{m-1} 
=
I_H +\sum_{k=1}^{m-1} x_k x_k^* + y_{m-1}^*y_{m-1},
\end{align*}
and
\begin{align*}
I_{H} 
& = \Phi(S_{e_1}) \Phi(S_{e_1})^* + x_1 x_1^* + \sum_{k=1}^{m-1} (\Phi(S_{e_k})^* \Phi(S_{e_k}) + y_k^*y_k) \\
& = 
\sum_{k=1}^{m} \Phi(P_{v_k}) + x_1x_1^* + \sum_{k=1}^{m-1} y_k^*y_k 
=
I_H + x_1x_1^* + \sum_{k=1}^{m-1} y_k^*y_k.
\end{align*}
It follows that $x_k=0$ and $y_k=0$ for all $ k \in [m-1]$, as required.

Next we show that $\G$ is ucp hyperrigid inside $\ca(G)$.
Let $\Phi \colon \ca(G) \to \B(H)$ be a faithful non-degenerate $*$-representation, and let $\phi_n \colon \B(H) \to \B(H)$ be unital completely positive maps such that
\[
\lim_n \phi_n(\Phi(S_{e_m})) = \Phi(S_{e_m})
\foral m \in \bN.
\]
We will show that
\[
\lim_n \phi_n(\Phi(c)) = \Phi(c)
\foral c \in \ca(\G).
\]
Towards this end, let us write
\[
1_m := \sum_{k=1}^m P_{v_k} \in \ca(G).
\]
For $m \in \bN$ we set
\[
H_m := [\Phi(1_m)H]
\qand
\Phi_m \colon \ca(G_m) \to \B(H_m); \Phi_m(c) = \Phi(c)|_{H_m}.
\]
By construction we have
\[
H = [\Phi(\ca(G))H] = \ol{\cup_m H_m}.
\]
Moreover we have that each $\Phi_m$ is unital satisfying 
\[
\Phi|_{\ca(G_m)} = \Phi_m \oplus 0_{H \ominus H_m}.
\]
Consider the completely contractive completely positive maps
\[
\psi_n^{(m)} \colon \B(H_m) \to \B(H_m); \psi_n^{(m)}(x) = P_{H_m} \phi_n(x \oplus 0_{H \ominus H_m}) |_{H_m},
\]
and observe that for each $ k \in [m-1]$ we get
\begin{align*}
\lim_n \psi_n^{(m)}(\Phi_m(S_{k}))
& =
\lim_n P_{H_m} \phi_n(\Phi_m(S_{k}) \oplus 0_{H \ominus H_m}) |_{H_m} \\
& =
\lim_n P_{H_m} \phi_n(\Phi(S_{k})) |_{H_m} 
=
P_{H_m} \Phi(S_k) |_{H_m}
=
\Phi_m(S_k).
\end{align*}
Since the set $\{S_k \mid  k \in [m-1]\}$ is ucp hyperrigid, by Corollary \ref{C:ccp/ucp} we have that it is also ccp hyperrigid and hence
\[
\lim_n P_{H_m} \phi_n(\Phi(c)) |_{H_m}
=
\lim_n \psi_n^{(m)}(\Phi_m(c)) 
= 
\Phi_m(c)
\foral c \in \ca(G_m).
\]
Since by construction we have $\ca(G_k) \subseteq \ca(G_m)$ for $2\leq k \leq m$ we get
\[
\lim_n P_{H_m} \phi_n(\Phi(c)) |_{H_m}
=
\Phi_m(c) 
\foral c \in \ca(G_k),  2\leq k \leq m.
\]
However for $2\leq k \leq m$ and $c \in \ca(G_k)$ we have
\[
\Phi_m(c) = \Phi_k(c) \oplus 0_{H_m \ominus H_k},
\]
and thus for fixed $k\geq 2$ we have
\[
\text{wot-}\lim_m \Phi_m(c) = \Phi(c)
\foral c \in \ca(G_k).
\]
Consequently we derive 
\begin{align*}
\lim_n \phi_n(\Phi(c))
& =
\lim_n \text{wot-}\lim_m P_{H_m} \phi_n(\Phi(c)) |_{H_m} \\
& =
\text{wot-}\lim_m \lim_n P_{H_m} \phi_n(\Phi(c)) |_{H_m} 
= 
\text{wot-}\lim_m \Phi_m(c) = \Phi(c),
\end{align*}
for any $c \in \ca(G_k)$, where we used that $\lim_n P_{H_m} \phi_n(\Phi(c)) |_{H_m} = \Phi_m(c)$ when interchanging the wot and norm limits.
As $k$ was arbitrary and $\cup_k \ca(G_k)$ is dense in $\ca(G)$, we get the required
\[
\lim_n \phi_n(\Phi(c)) = \Phi(c)
\foral c \in \ca(G).
\]

Finally we have that $\G$ is not separating.
This follows from \cite[Theorem 5.1]{Sal17} or \cite[Remark 5.2]{Sal17} as $\G$ does not contain all projections from the vertices.
In particular for the gauge action $\{\be_z\}_{z \in \bT}$ on the graph C*-algebra $\ca(G)$ and any vertex $v$ of $G$ we can define the state
\[
\vphi_v \colon \ca(G) \to \bC P_v; \vphi_v(c) = P_v (\int_{\bT} \be_z(c) d\la) P_v.
\]
Since the set $\{S_{e_m}\}_{m \in \bN}$ is not contained in the fixed point algebra we have $\vphi_v(S_{e_m}) = 0$ for all $m \in \bN$.
However we have $\vphi_v(P_v) = P_v$, and so $\vphi_v \neq 0$.

Note that the finite graphs $G_m$ give more examples of non-separating ucp hyperrigid non-unital sets generating unital C*-algebras. Indeed, we have that $\G_m$ generates a unital C*-algebra, it is ucp hyperrigid, and it is not separating as the induced state ${\vphi_v}|_{\ca(G_m)}$ for any  vertex $v$ of $G_m$ annihilates $\G_m$.
\end{example}

\section{Hyperrigidity for non-degenerate C*-correspondences} \label{S:cor}

\subsection{Preliminaries on C*-correspondences}

The theory of (right) Hilbert C*-modules is well-developed.
The reader is addressed to \cite{Lan95,MT05} for an excellent introduction to the subject.

A \emph{C*-correspondence} $X$ over a C*-algebra $A$ is a right Hilbert C*-module over $A$ with a left action given by a $*$-homomorphism $\vphi_X \colon A \to \L(X)$, where $\L(X)$ denotes the C*-algebra of adjointable operators on $X$. A C*-correspondence $X$ over $A$ is called \emph{non-degenerate} if $[\vphi_X(A) X] = X$.
We write $\K(X)$ for the closed linear span of the rank one adjointable operators $\theta_{\xi, \eta}$.
For two C*-corresponden\-ces $X, Y$ over the same $A$ we write $X \otimes_A Y$ for the stabilised tensor product over $A$.
For a C*-correspondence $X$ over $A$ and $n \geq 2$ we write $X^{\otimes n} = X \otimes_A X^{\otimes n-1}$ which becomes a C*-correspondence over $A$ by the action $\vphi_{X^{\otimes n}} := \vphi_X \otimes \id_{X^{\otimes n-1}}$.
We write $X^{\otimes 0}$ to denote the trivial C*-correspondence $A$ over itself.

A representation of a C*-correspondence $X$ over $A$ is a triple $(H,\pi,t)$ where $\pi \colon A\to \B(H)$ is a $*$-homomorphism and $t \colon X\to \B(H)$ is a linear map such that for all $\xi,\eta \in X$ and $a\in A$ we have
\[
\pi(a)t(\xi)=t(\varphi_X(a)\xi)
\qand
t(\xi)^* t(\eta) = \pi(\sca{\xi,\eta}).
\]
It then follows that $t(\xi) \pi(a) = t(\xi a)$ for all $a \in A$ and $\xi \in X$.  
We will simply write $(\pi,t)$ when $H$ is understood or when $\pi$ and $t$ take values inside a C*-algebra. 
Every representation $(\pi,t)$ as above defines a $*$-representation 
\[
\psi_t \colon \K(X)\to \B(H); \theta_{\xi, \eta} \mapsto t(\xi) t(\eta)^*
\foral \xi, \eta \in X.
\] 
If $\pi$ is injective, then both $t$ and $\psi_t$ are isometric.
We write $\ca(\pi,t)$ (resp.\ $\ol{\text{alg}}(\pi,t)$) for the C*-subalgebra (resp.\ norm-closed subalgebra) of $\B(H)$ generated by $\pi(A)$ and $t(X)$.
We say that a representation $(\pi,t)$ of $X$ \emph{admits a gauge action} if there exist $*$-automorphisms $\{\be_z\}_{z\in\bT}$ of $\ca(\pi,t)$ such that for all $z \in \bT$ we have
\[
\be_z(\pi(a))=\pi(a) \foral a \in A \qand \be_z(t(\xi))=zt(\xi) \foral \xi \in X.
\]

We will write $(\ol{\pi}, \ol{t})$ for the \emph{Fock representation} on $\F(X) = \sum_{n \geq 0} X^{\otimes n}$ given by
\[
\ol{\pi}(a) \xi_n = \vphi_{X^{\otimes n}}(a) \xi_n
\qand
\ol{t}(\xi) \xi_n = \xi \otimes \xi_n
\]
for every $\xi_n \in X^{\otimes n}$, $n \in \bN$, and $a \in A$, $\xi \in X$.
By \cite{Kat04} the \emph{Toeplitz-Pimsner} C*-algebra $\T_X := \ca(\ol{\pi}, \ol{t})$ is universal with respect to the representations of $X$.
Given a representation $(\pi,t)$ of $X$, we write
\[
\pi\times t\colon\T_X\to\ca(\pi,t)
\]
for the (unique) canonical $*$-epimorphism guaranteed by the universal property of $\T_X$. 
The Gauge-Invariant Uniqueness Theorem for $\T_X$ asserts that $\pi \times t$ is faithful if and only if $\ca(\pi,t)$ admits a gauge action and $\pi(A) \cap \psi_t(\K(X)) = \{0\}$, see \cite[Theorem 6.2]{Kat04}.

There is a significant quotient of $\T_X$ that gives rise to several known C*-constructions such as graph C*-algebras, crossed products by $\bZ$ and C*-algebras associated with topological graphs.
Let $J_X := \ker\vphi_X^\perp \cap \phi^{-1}(\K(X))$ be Katsura's ideal.
A representation $(\pi,t)$ of $X$ is called \emph{Cuntz-Pimsner covariant} if it satisfies
\[
\pi(a)= \psi_t(\vphi_X(a)) \foral a \in J_X.
\]
Equivalently, if $(\pi,t)$ acts on a Hilbert space $H$ and $p$ is the projection on the space $[\psi_t(\K(X))H]$, then we have that $(\pi,t)$ is Cuntz-Pimsner covariant if and only if
\[
\pi(a)(I_H - p) = 0 \foral a \in J_X.
\]
Indeed, this follows because, if $(k_\la)_{\la} \subseteq \K(X)$ is an approximate unit and $a \in \vphi_X^{-1}(\K(X))$, then
\begin{align*}
\pi(a) p 
& = \text{wot-}\lim_\la \pi(a) \psi_t(k_\la) 
= \text{wot-}\lim_\la \psi_{t}(\vphi_X(a) k_\la) \\
&=
\text{$\nor{\cdot}$-}\lim_\la \psi_{t}(\vphi_X(a) k_\la) 
= \psi_{t}(\vphi_X(a)).
\end{align*}

We write $\O_X$ for the universal C*-algebra with respect to the Cuntz-Pimsner covariant representations, and refer to $\O_X$ as the \emph{Cuntz-Pimsner algebra of $X$}.
We will write $(\pi_X, t_X)$ for the universal Cuntz-Pimsner covariant representation of $X$. 
By \cite[Proposition 7.14]{Kat07} the C*-algebra $\O_X$ is co-universal with respect to the representations of $X$ that are faithful on $A$ and admit a gauge action.
The Gauge-Invariant Uniqueness Theorem for $\O_X$ then follows: a $*$-representation of $\O_X$ is faithful if and only if it admits a gauge action and it is faithful on $\pi_X(A)$, see \cite[Theorem 6.4]{Kat04}.
We next record a description of non-degenerate C*-correspondences for future reference.

\begin{lemma}\label{L:nondeg}
Let $X$ be a C*-correspondence over a C*-algebra $A$. 
The following are equivalent:
\begin{enumerate}
\item $X$ is non-degenerate.
\item For every $(H,\pi,t)$ representation of $X$ we have 
\[
[\pi(A)H]=[\ca(\pi,t)H].
\]
\end{enumerate}
In particular, if $X$ is non-degenerate and $(H,\pi,t)$ is a representation of $X$, then $\pi$ is non-degenerate if and only if $\pi\times t$ is non-degenerate.
\end{lemma}

\begin{proof}
If $X$ is non-degenerate, then $[\vphi_X(A)X] = X$, and thus an approximate unit of $A$ defines an approximate unit for $\T_X$.
Therefore, if $(H,\pi,t)$ is a representation of $X$, then $A$ defines an approximate unit for $\ca(\pi,t)$, and thus $[\pi(A)H]=[\ca(\pi,t)H]$.

Conversely, an application of Proposition \ref{P:multiplier} for $\ol{\pi}(A) \subseteq \T_X$ yields $[\ol{\pi}(A) \T_X] = \T_X$. 
Hence $(\ol{\pi}(a_\la))_\la$ is an approximate unit for $\T_X$ for an approximate unit $(a_\la)$ of $A$; in particular 
\[
\ol{t}(\xi) = \lim_\la \ol{\pi}(a_\la) \ol{t}(\xi) = \lim_\la \ol{t}(\vphi_X(a_\la)\xi) \foral \xi \in X.
\]
Since $\ol{t}$ is isometric we obtain $\xi = \lim_\la \vphi_X(a_\la)\xi$ for all $\xi \in X$, i.e., $X$ is non-degenerate, and the proof is complete.
\end{proof}

We refer to $\T_X^+ := \ol{\text{alg}}(\ol{\pi},\ol{t})$ as the \emph{tensor algebra of $X$}.
By \cite[Lemma 3.5]{KK06} we have $\T_X^+ \simeq \ol{\text{alg}}(\pi_X, t_X)$, and by \cite[Theorem 3.7]{KK06} we also have  $\cenv(\T_X^+) \simeq \O_X$. 
Kim \cite{Kim21} considers the selfadjoint operator space $\fS(A,X)$ generated by $\ol{\pi}(A)$ and $\ol{t}(X)$ inside $\T_X$ and the selfadjoint operator space $S(A,X)$ generated by $\pi_X(A)$ and $t_X(X)$ inside $\O_X$. 
In \cite[Proposition 2.5]{Kim21} it is shown that $\fS(A,X)$ is completely isometric completely order isomorphic to $S(A,X)$ by passing to unitisations.
This can be shown directly at their level as the following alternative proof shows.

\begin{proposition}\label{P:S is fS} \cite[Proposition 2.5]{Kim21}.
Let $X$ be a C*-correspondence over a C*-algebra $A$. 
Then $\fS(A,X)$ and $S(A,X)$ are canonically completely isometric completely order isomorphic. 
\end{proposition}

\begin{proof}
Let $(\wt{\pi},\wt{t}) := (\pi_X \otimes I, t_X \otimes V)$ be the injective representation of $X$ in $\O_X\otimes \B(\ell^2(\bZ_+))$, where $V$ is the unilateral shift of $\ell^2(\bZ_+)$. 
An application of the Gauge-Invariant Uniqueness Theorem for $\T_X$ gives an injective $*$-homomorphism $\wt{\pi}\times\wt{t} \colon \T_X \to \O_X\otimes \B(\ell^2(\bZ_+))$ that preserves generators of the same index.
An application of the Gauge-Invariant Uniqueness Theorem for $\O_X$ provides an injective $*$-homomorphism 
\[
\de \colon \O_X\to \O_X\otimes \B(\ell^2(\bZ)); t_X(\xi) \mapsto t_X(\xi)\otimes U\text{ and } \pi_X(a)\mapsto \pi_X(a)\otimes I,
\]
where $U$ is the bilateral shift of $\ell^2(\bZ)$, see for example \cite[Remark 2.2 (4)]{LPRS87}. 
For the compression $\phi$ to $\ell^2(\bZ_+)$, we obtain a completely contractive completely positive map 
\[
(\id \otimes \phi) \circ \de \colon \O_X\to \O_X\otimes \B(\ell^2(\bZ_+)); t_X(\xi) \mapsto \wt{t}(\xi)\text{ and } \pi_X(a)\mapsto \wt{\pi}(a).
\]
The map $(\wt{\pi}\times\wt{t})^{-1} \circ (\id \otimes \phi) \circ \de |_{S(A,X)}$ is a completely contractive completely positive inverse of the completely contractive completely positive map $(\pi_X\times t_X)|_{\fS(A,X)}$, and the proof is complete.
\end{proof}

In \cite[Proposition 2.6]{Kim21} Kim shows that, when $X$ is non-degenerate, the C*-algebra generated by $S(A,X)$ inside $\cenv(S(A,X)^\#)$ is $*$-isomorphic to $\O_X$.
The argument starts by assuming the existence of a $*$-epimorphism from $\O_X$ inside this C*-algebra; this requires showing that the inclusion of $S(A,X)$ in $\O_X$ is an embedding.
Proposition \ref{P:ext} closes this gap, and the complete argument is as follows.

\begin{proposition}\label{P:cenv S(A,X)} \cite[Proposition 2.6]{Kim21}. 
Let $X$ be a non-degenerate C*-correspondence. 
Then the C*-envelope of $\fS(A,X)$ is $\O_X$.
\end{proposition}

\begin{proof}
By Proposition \ref{P:S is fS} it suffices to prove that the C*-envelope of $S(A,X)$ is $\O_X$. 
Since $X$ is non-degenerate, any approximate unit for $A$ is also an approximate unit for $\O_X$, and hence by Proposition \ref{P:ext} we have that the inclusion $S(A,X)\subseteq \O_X$ is an embedding. 
Therefore, we obtain a $*$-epimorphism $\Phi\colon \O_X\to \cenv(S(A,X))$ which is completely isometric on $S(A,X)$. 
Note that $\cenv(S(A,X))$ admits a gauge action. Indeed, if $\{\be_z\}_{z \in \bT}$ is the gauge action of $\O_X$, then each $\be_z|_{S(A,X)}\colon S(A,X)\to S(A,X)$ is a completely isometric complete order isomorphism and hence extends to a $*$-automorphism of $\cenv(S(A,X))$.
Since $A$ embeds faithfully in $\cenv(S(A,X))$, the co-universal property of $\O_X$ from \cite[Proposition 7.14]{Kat07} implies that $\Phi$ is a $*$-isomorphism.
\end{proof}

\subsection{Hyperrigidity}

In this subsection we show that the results of \cite{KR20, KR21} and \cite{Kim21} match for non-degenerate C*-correspondences, and we thus answer a question of \cite{KR20}.
This was only verified for C*-correspondences arising from topological graphs with an open range map in \cite{Kim21}, and in \cite{KR20} without any assumption on the range map. 
From now on we will refer to ccp hyperrigidity and/or ucp hyperrigidity simply as hyperrigidity, them being equivalent by Corollary \ref{C:ccp/ucp}.
Our goal is to complete the following equivalences.

\begin{theorem}\label{T:hyp}
Let $X$ be a non-degenerate C*-correspondence over a C*-algebra $A$.
The following are equivalent:
\begin{enumerate}
\item $\pi_X(A) \cup t_X(X)\subseteq \O_X$ is hyperrigid.
\item $(\pi_X \times t_X)(\fS(A,X)) \subseteq \O_X$ is hyperrigid.
\item $(\pi_X \times t_X)(\T_X^+) \subseteq \O_X$ is hyperrigid.
\item $[\pi(J_X)H] = [\psi_t(\K(X))H]$ for every Cuntz-Pimsner covariant representation $(H,\pi,t)$ of $X$ with $(H,\pi)$ being non-degenerate.
\item $[t(\vphi_X(J_X)X)H] = [t(X)H]$ for every representation $(H,\pi,t)$ of $X$ with $(H,\pi)$ being non-degenerate.
\item $[\vphi_X(J_X)X] \otimes_\si H = X \otimes_\si H$ for every non-degenerate representation $(H,\si)$ of $A$.
\item $[\vphi_X(J_X) \K(X)]  =\K(X)$.
\item $[\vphi_X(J_X) X] = X$.
\item $\vphi_X(J_X) X = X$.
\end{enumerate}
\end{theorem}

For the proof of Theorem \ref{T:hyp} we require several intermediate results.
The next lemma and its proof appear in \cite[Lemma 2.9]{KR20}.

\begin{lemma} \label{L:inv} \cite[Lemma 2.9]{KR20}.
Let $X$ be a C*-correspondence over a C*-algebra $A$ and let $(H, \pi, t)$ be a representation of $X$.
If $K \subseteq H$ is an invariant subspace for $\pi(A)$ and $t(X)$, then $(K, \pi|_K, t|_K)$ is a representation of $X$.
\end{lemma}

\begin{proof}
Since $\pi$ is a $*$-representation, then $p_K \in \pi(A)'$.
Hence
\[
\pi(a) p_K t(\xi)p_K = \pi(a) t(\xi)p_K = t(\vphi_X(a) \xi)p_K
\]
for every $a \in A$ and $\xi \in X$.
Moreover we get
\[
p_K t(\xi)^* p_K t(\eta) p_K
=
p_K t(\xi)^* t(\eta) p_K
=
p_K \pi(\sca{\xi,\eta})p_K
=
\pi(\sca{\xi,\eta})p_K
\]
for every $\xi, \eta \in X$, as required.
\end{proof}

The next lemma and its proof appear in \cite[proof of Claim in Theorem 3.5]{KK12}.

\begin{lemma} \label{L:dilation} \cite[proof of Claim in Theorem 3.5]{KK12}.
Let $X$ be a C*-correspon\-dence over a C*-algebra $A$.
If $(\pi,t)$ and $\left(\wt{\pi},\wt{t}\right)$ are representations of $X$ such that $(\wt{\pi}\times\wt{t})|_{\ol{\pi}(A) \cup \ol{t}(X)}$ is a dilation of $(\pi\times t)|_{\ol{\pi}(A) \cup \ol{t}(X)}$, then 
\[
\wt{\pi}(\cdot)=
\begin{pmatrix}
\pi(\cdot) & 0 \\
0 & *
\end{pmatrix}
\qand
\wt{t}(\cdot)=
\begin{pmatrix}
t(\cdot) & * \\
0 & *
\end{pmatrix}.
\]
\end{lemma}

\begin{proof}
Since $\pi$ is a $*$-representation, and a compression of the $*$-representation $\wt{\pi}$, it follows that $\pi$ is a direct summand of $\wt{\pi}$.
Next, for $\xi\in X$ we can write 
\[
\wt{t}(\xi)
=
\begin{pmatrix}
t(\xi) & \ast \\
x & \ast
\end{pmatrix}.
\]
Therefore, we obtain
\begin{align*}
\begin{pmatrix}
t(\xi)^*t(\xi)+x^*x & \ast \\
\ast & \ast
\end{pmatrix}
&=
\wt{t}(\xi)^*\wt{t}(\xi)=\wt{\pi}(\sca{\xi,\xi})\\
&=
\begin{pmatrix}
\pi(\sca{\xi,\xi})& 0 \\
0 & \ast
\end{pmatrix}
=
\begin{pmatrix}
t(\xi)^* t(\xi) & 0 \\ 0 & \ast
\end{pmatrix}.
\end{align*}
By equating the (1,1)-entries we obtain $x=0$ as required.
\end{proof}

The next lemma and its proof appear in \cite[proof of Theorem 3.1]{KR20} and in \cite[Lemma 3.3]{Kim21}.

\begin{lemma}\label{L:Kasparov} \cite[proof of Theorem 3.1]{KR20}, \cite[Lemma 3.3]{Kim21}.
Let $X$ be a C*-correspondence over a C*-algebra $A$ and $k\in \K(X)$. 
Then there exists a sequence $(x_n)_{n\in\bN} \subseteq X$ such that $\sum_{n=1}^{\infty}\theta_{x_n,x_n} k = k$, where the series convergence is in the norm topology.
\end{lemma}

\begin{proof}
By considering a sequence of finite sums of rank one operators we have that $k$ is in the module their vectors generate. 
As this forms a countably generated module whose compacts contain $k$, the statement then follows from Kasparov's Stabilisation Theorem.
\end{proof}

The next proposition appears in \cite[proof of Theorem 3.1]{KR21} and in \cite[proof of Theorem 3.5]{Kim21}.

\begin{proposition}\label{P:J part} \cite[proof of Theorem 3.1]{KR21}, \cite[proof of Theorem 3.5]{Kim21}.
Let $X$ be a C*-correspondence over a C*-algebra $A$. 
If $(\pi,t)$ and $\left(\wt{\pi},\wt{t}\right)$ are Cuntz-Pimsner covariant representations of $X$ such that $(\wt{\pi}\times\wt{t})|_{\ol{\pi}(A) \cup \ol{t}(X)}$ is a dilation of $(\pi\times t)|_{\ol{\pi}(A) \cup \ol{t}(X)}$, then 
\[
\wt{t}(\vphi_X(a)\xi)=
\begin{pmatrix}
t(\vphi_X(a)\xi) & 0 \\
0 & *
\end{pmatrix}, \foral a\in J_X , \xi\in X.
\]
\end{proposition}

\begin{proof}
As $J_X$ is spanned by its positive elements, it suffices to show the statement for $a \geq 0$ in $J_X$ and $\xi \in X$.
Towards this end, pick a sequence $(x_n)_n \subseteq X$ such that 
\[
\sum_{n=1}^\infty \theta_{x_n, x_n} \vphi_X(a)^{1/2} = \vphi_X(a)^{1/2}
\]
provided by Proposition \ref{L:Kasparov}. 
Then we have
\[
\sum_{n=1}^\infty \theta_{\vphi_X(a)^{1/2}x_n, \vphi_X(a)^{1/2}x_n} = \vphi_X(a)^{1/2} \sum_{n=1}^\infty \theta_{x_n, x_n} \vphi_X(a)^{1/2} = \vphi_X(a).
\]
Hence without loss of generality, we may pick a sequence $(x_n)_n \subseteq X$ such that $\sum_{n=1}^\infty \theta_{x_n, x_n} = \vphi_X(a)$.
By using Lemma \ref{L:dilation} let us write
\[
\wt{t}(\xi) = \begin{pmatrix} t(\xi) & y \\ 0 & \ast \end{pmatrix}
\qand
\wt{t}(x_n) = \begin{pmatrix} t(x_n) & y_n \\ 0 & \ast \end{pmatrix} \foral n \in \bN.
\]
We then have 
\begin{align*}
\begin{pmatrix} t(x_n)^* t(\xi) & t(x_n)^* y \\ y_n^* t(\xi) & \ast 
\end{pmatrix} 
&= 
\wt{t}(x_n)^* \wt{t}(\xi) 
= 
\wt{\pi}(\sca{x_n,\xi})\\
&= \begin{pmatrix} 
\pi(\sca{x_n,\xi}) & 0 \\ 0 & \ast \end{pmatrix} 
=
\begin{pmatrix}
t(x_n)^* t(\xi) & 0 \\ 0 & \ast \end{pmatrix}.
\end{align*}
By equating the $(1,2)$-entries we derive that $t(x_n)^* y = 0$ for all $n \in \bN$.
We now have
\[
\wt{t}(\vphi_X(a) \xi) = \wt{\pi}(a) \wt{t}(\xi) = \begin{pmatrix} \pi(a) t(\xi) & \pi(a) y \\ 0 & \ast \end{pmatrix}.
\]
However, since $a \in J_X$ and $(\pi,t)$ is Cuntz-Pimsner covariant, we get
\[
\pi(a) y = \psi_{t}(\vphi_X(a)) y = \sum_{n=1}^\infty t(x_n) t(x_n)^* y = 0, 
\]
and the proof is complete.
\end{proof}

The next lemma is a consequence of the stabilised tensor product of Hilbert C*-modules.

\begin{lemma} \label{L:equi conv}
Let $X$ be a C*-correspondence over a C*-algebra $A$ and let $J$ be an ideal in $A$.
The following are equivalent:
\begin{enumerate}
\item $[t(\vphi_X(J)X)H] = [t(X)H]$ for any representation $(H,\pi,t)$ of $X$.
\item $[\vphi_X(J)X] \otimes_\si H = X \otimes_\si H$ for any representation $(H,\si)$ of $A$.
\end{enumerate}
\end{lemma}

\begin{proof}
Suppose that item (i) holds. 
For a $*$-representation $(H,\si)$ of $A$, consider the representation $(\pi,t) := (\ol{\pi} \otimes I, \ol{t} \otimes I)$ on $\F(X) \otimes_\si H$. 
We then have
\begin{align*}
\sum_{n = 1}^\infty X^{\otimes n} \otimes_\si H
& =
[t(X) (\F(X) \otimes_\si H)] \\
& = 
[t(\vphi_X(J) X) (\F(X) \otimes_\si H)] \\
& =
\sum_{n = 1}^\infty [\vphi_X(J)X] \otimes_A (X^{\otimes n-1} \otimes_\si H),
\end{align*}
where we use the assumption in the second equality.
By taking the projection to the first summand we derive the required $X \otimes_\si H = [\vphi_X(J)X] \otimes_\si H$.

For the converse, let $(H, \pi, t)$ be a representation of $X$.
Since both $\pi$ and $t$ are completely contractive, there is a well-defined isometric map
\[
W \colon X \otimes_\pi H \to H; \xi \otimes h \mapsto t(\xi)h.
\] 
Therefore we obtain
\[
[t(\vphi_X(J) X)H] 
= 
[W( [\vphi_X(J)X)] \otimes_\pi H)]
=
[W( X \otimes_\pi H)]
=
[t(X) H],
\]
where we used the assumption for $(H,\pi)$ in the second equality, and the proof is complete.
\end{proof}

Let us now return to Theorem \ref{T:hyp} and make the following remark on non-degeneracy of the $*$-representations of $A$.

\begin{remark} \label{R:remnd}
By virtue of Lemma \ref{L:nondeg}, if $X$ is non-degenerate then all elements in any $\ca(\pi,t)$ are supported on $[\pi(A)H]$ for a representation $(H,\pi,t)$ of $X$.
Moreover non-degeneracy of both the left and the right action of $A$ implies
\[
[t(X)H] = [t(X) \pi(A) H]
\qand
[\psi_t(\K(X))H] = [\psi_t(\K(X)) \pi(A) H].
\]
Moreover by the properties of the stabilised tensor product and since $X = [X \cdot A]$ we have
\[
[\vphi_X(J)X] \otimes_\si H = [\vphi_X(J)X] \otimes_\si [\si(A)H]
\qand
X \otimes_\si H = X \otimes_\si [\si(A)H].
\]
Hence Theorem \ref{T:hyp} is still valid after removing the non-degeneracy clause on $(H,\pi)$ in items (iv) and (v), and on $(H,\si)$ in item (vi).
\end{remark}

We now give the proof of Theorem \ref{T:hyp}.

\begin{proof}[{\bf Proof of Theorem \ref{T:hyp}}]

\noindent
[(i) $\Leftrightarrow$ (ii)]. 
By Proposition \ref{P:unitcls} and Theorem \ref{T:hypiffuep}, Lemma \ref{L:iffmaxspace} entails that $S(A,X)$ is hyperrigid if and only if $\pi_X(A) \cup t_X(X)$ is hyperrigid.

\medskip

\noindent
[(i) $\Leftrightarrow$ (iii)]. 
In view of Lemma \ref{L:dilation}, by applying Lemma \ref{L:iffmaxalg}, and Proposition \ref{P:unitcls} and Theorem \ref{T:hypiffuep}, we obtain that $\pi_X(A) \cup t_X(X)$ is hyperrigid if and only if the algebra $(\pi_X \times t_X)(\T_X^+)$ that it generates is hyperrigid.

\medskip

\noindent
[(iii) $\Rightarrow$ (iv)]. 
This is proven in \cite[Theorem 2.7]{KR20}.
Suppose that the generating set $(\pi_X \times t_X)(\T_X^+) \subseteq \O_X$ is hyperrigid and let $(H,\pi,t)$ be a Cuntz-Pimsner covariant representation of $X$ where $\pi\times t$ is non-degenerate.
Then $\pi$ acts non-degenerately on $H$ by Lemma \ref{L:nondeg}.
For the bilateral shift $U \in \B(\ell^2(\bZ))$, consider the representation $(\wt{\pi}, \wt{t}) := (\pi \otimes I, t \otimes U)$ on $H \otimes \ell^2(\bZ)$, which in turn is a Cuntz-Pimsner covariant representation of $X$ where $\wt{\pi}\times \wt{t}$ is non-degenerate.
Set
\[
M:= [\psi_t(\K(X))H] \ominus [\pi(J_X)H] 
\qand
\wh{M} := \sum_{n = 0}^\infty [t(X)^n M] \otimes e_n \subseteq H \otimes \ell^2(\bZ).
\]
To reach a contradiction suppose that $M \neq (0)$, and define $(\wh{\pi}, \wh{t})$ on $\wh{M}$ by
\begin{align*}
& \wh{\pi} \colon A \to \B(\wh{M}); \wh{\pi}(a) = (\pi(a) \otimes I)|_{\wh{M}} \foral a \in A, \\
& \wh{t} \colon X \to \B(\wh{M}); \wh{t}(\xi) = (t(\xi) \otimes U)|_{\wh{M}} \foral \xi \in X.
\end{align*}
Then $(\wh{M}, \wh{\pi}, \wh{t})$ is a Cuntz-Pimsner covariant representation of $X$ where $\wh{\pi}\times \wh{t}$ is non-degenerate.

Indeed, by Lemma \ref{L:inv} it follows that $(\wh{M}, \wh{\pi}, \wh{t})$ is a representation of $X$. 
Next, since $X$ is non-degenerate, then $\pi(A)$ acts non-degenerately on the space $[\psi_t(\K(X))H]$.
Moreover $\pi(A)$ acts non-degenerately on $[\pi(J_X)H]$ as $J_X$ is an ideal in $A$.
Hence $[\pi(A) M] = M$.
Since $X$ is non-degenerate we also have $[\pi(A)t(X)^nH] = [t(X)^nH]$ for every $n \in \bN$; hence $\wh{\pi}(A)$ acts non-degenerately on $\wh{M}$.

For Cuntz-Pimsner covariance, let $\wh{p}$ be the orthogonal projection on the subspace
\[
[\psi_{\wh{t}}(\K(X)) \wh{M}] = \sum_{n \geq 1} [t(X)^n M] \otimes e_n.
\]
A straightforward computation for $a \in J_X$ gives
\[
\wh{\pi}(a)(I_{\wh{M}} - \wh{p}) = \wt{\pi}(a)|_{M\otimes e_0} = 0,
\]
since by definition $[\pi(J_X) M] = (0)$, and thus $(\wh{\pi},\wh{t})$ is Cuntz-Pimsner covariant.

By construction we have that $(\wt{\pi} \times \wt{t})|_{\T_X^+}$ is a dilation of $(\wh{\pi} \times \wh{t})|_{\T_X^+}$.
Since $\T_X^+$ is assumed to be hyperrigid in $\O_X$, by Theorem \ref{T:hypiffuep} we have that $\wh{M}$ is a reducing subspace of $\wt{\pi} \times \wt{t}$.
Let us write $P_{\wh{M}}$ for the projection on $\wh{M} \subseteq H \otimes \ell^2(\bZ)$, so that
\[
\wt{t}(\xi)^* P_{\wh{M}} = P_{\wh{M}} \wt{t}(\xi)^* \foral \xi \in X.
\]
In particular for every $h \in M$ and $\xi \in X$ we obtain
\begin{align*}
t(\xi)^* h \otimes e_{-1}
& =
\wt{t}(\xi)^* (h \otimes e_{0})
=
\wt{t}(\xi)^* P_{\wh{M}} (h \otimes e_{0}) \\
& = 
P_{\wh{M}} \wt{t}(\xi)^* (h \otimes e_0)
=
P_{\wh{M}} (t(\xi)^*h \otimes e_{-1})
=
0.
\end{align*}
Therefore we have $t(X)^* h = (0)$ and thus $\psi_t(k) h = 0$ for every $k \in \K(X)$.
Since we have assumed that $h \in M \subseteq [\psi_t(\K(X))H]$, considering an approximate unit of $\K(X)$ implies $h =0$, leading to the contradiction that $M = (0)$.

\medskip

\noindent
[(iv) $\Rightarrow$ (v)]. 
Let $(u_\la)_\la$ be an approximate unit of $J_X$ and let $(H,\pi,t)$ be a Cuntz-Pimsner covariant representation of $X$ where $\pi\times t$ is non-degenerate.
If item (iv) holds, then $\lim_\la \pi(u_\la) h = h$ for every $h \in [\psi_t(\K(X))H]$.
Since $\pi(J_X) \subseteq \psi_t(\K(X))$ we have
\[
[\pi(J_X) H'] = H'= [\psi_t(\K(X)) H']
\]
whenever $H'= [\psi_t(\K(X))L]$ for some subspace $L\subseteq H$.

Now consider $(H,\pi, t)$ to be just a  representation of $X$ where $\pi\times t$ is non-degenerate, equivalently $(H,\pi)$ is non-degenerate by Lemma \ref{L:nondeg}.
Since $X$ and $\pi$ are non-degenerate we either have that $\T_X^+$ and $(\pi \times t)|_{\T_X^+}$ are unital or $\T_X^+$ is not unital. Thus by \cite[Theorem 1.2]{DM05} and \cite[Proposition 2.5]{DS18} if necessary, we obtain that in both cases $(\pi \times t)|_{\T_X^+}$ admits a maximal dilation. 
Since $\cenv(\T_X^+)=\O_X$, such a maximal dilation is a Cuntz-Pimsner covariant representation $(\wt{H},\wt{\pi},\wt{t})$. 
By considering its non-degenerate compression, we may assume that $\wt{\pi}\times \wt{t}$ is non-degenerate and $(\wt{\pi} \times \wt{t})|_{\T_X^+}$ dilates $(\pi \times t)|_{\T_X^+}$.
By Lemma \ref{L:dilation} we have
\[
\left[\wt{t}(X) \begin{pmatrix} H \\ 0 \end{pmatrix}
\right]
=
\left[\begin{pmatrix} t(X) & * \\  0 & * \end{pmatrix}
\begin{pmatrix} H \\ 0 \end{pmatrix}\right]
=
\begin{pmatrix} [t(X)H] \\  0 \end{pmatrix}.
\]
Since $[\K(X) X] = X$, we have
\begin{align*}
\left[\wt{t}(X) \begin{pmatrix} H \\ 0 \end{pmatrix} \right]
& =
\left[\psi_{\wt{t}}(\K(X))\wt{t}(X) \begin{pmatrix} H \\ 0 \end{pmatrix} \right] 
=
\left[\wt{\pi}(J_X) \psi_{\wt{t}}(\K(X)) \wt{t}(X) \begin{pmatrix} H \\ 0 \end{pmatrix} \right] \\
& =
\left[\wt{\pi}(J_X) \wt{t}(X) \begin{pmatrix} H \\ 0 \end{pmatrix} \right] 
=
\left[\wt{t}(\vphi_X(J_X)X) \begin{pmatrix} H \\ 0 \end{pmatrix} \right] \\
& =
\left[\begin{pmatrix} t(\vphi_X(J_X)X) & * \\ 0 & * \end{pmatrix}
\begin{pmatrix} H \\  0 \end{pmatrix} \right]
=
\begin{pmatrix} [t(\vphi_X(J_X)X) H] \\ 0 \end{pmatrix}.
\end{align*}
We conclude $[t(X)H] = [t(\vphi_X(J_X)X)H]$, as required.

\medskip

\noindent
[(v) $\Rightarrow$ (iv)]. 
This is immediate, since $\pi(J_X) \subseteq \psi_t(\K(X))$ for every $(H,\pi,t)$ Cuntz-Pimsner covariant representation, and by using that $[\K(X) X] = X$.

\medskip

\noindent
[(iv) $\Rightarrow$ (iii)]. 
Let $(H,\pi,t)$ be a Cuntz-Pimsner covariant representation of $X$ where $\pi\times t$ is non-degenerate.
By Theorem \ref{T:hypiffuep} we have to show that if  $(\wt{H}, \wt{\pi}, \wt{t})$ is a Cuntz-Pimsner covariant representation such that $(\wt{\pi} \times \wt{t})|_{\T_X^+}$ dilates $(\pi \times t)|_{\T_X^+}$, then $H$ is reducing for $\wt{\pi} \times \wt{t}$.
By passing to the non-degenerate compression we may assume that $\wt{\pi} \times \wt{t}$ is non-degenerate.
By Lemma \ref{L:dilation} and Proposition \ref{P:J part} we have that $\pi(A)$ and $t(\vphi_X(J_X)X)$ are direct summands of $\wt{\pi}(A)$ and $\wt{t}(\vphi_X(J_X)X)$, respectively.
Let $(u_\la)_\la$ be an approximate unit of $J_X$.
By assumption we have
\[
\textup{wot-}\lim_\la \pi(u_\la) = P_{[\psi_t(\K(X))H]}.
\]
Hence for any $\xi \in X$ we have
\[
t(\xi) = P_{[\psi_t(\K(X))H]} t(\xi) = \textup{wot-}\lim_\la \pi(u_\la) t(\xi) = \textup{wot-}\lim_\la t(\vphi_X(u_\la) \xi),
\]
i.e., $t(X)$ is in the wot-closure of $[t(\vphi_X(J_X)X)]$.
Likewise we have that $\wt{t}(X)$ is in the wot-closure of $[\wt{t}(\vphi_X(J_X)X)]$.
Hence by Proposition \ref{P:a-max} we have that $t(X)$ is a direct summand of $\wt{t}(X)$ as well, and thus $H$ is reducing for $\wt{\pi} \times \wt{t}$ as required.

\medskip

\noindent
[(v) $\Leftrightarrow$ (vi)]. 
This follows from Lemma \ref{L:equi conv}.

\medskip

\noindent
[(iv) $\Leftrightarrow$ (vii)]. 
This follows from Proposition \ref{P:multiplier} since
\[
\pi_X(J_X) = \psi_{t_X}(\vphi_X(J_X)) \subseteq \psi_{t_X}(\K(X)) \subseteq \O_X,
\]
for the universal $*$-representation $(\pi_X, t_X)$ of $\O_X$.

\medskip

\noindent
[(vii) $\Leftrightarrow$ (viii)]. 
This is immediate by the TRO property $[\K(X) X] = X$.

\medskip

\noindent
[(viii) $\Leftrightarrow$ (ix)]. 
This follows by an application of the Hewitt-Cohen Factorisation Theorem.

\medskip

\noindent
[(ii) $\Leftrightarrow$ (ix)]. 
This is \cite[Theorem 3.5]{Kim21}.
Alternatively it follows from [(iii) $\Leftrightarrow$ (ix)] since item (ii) is now shown to be equivalent to item (iii).
\end{proof}

\begin{remark} \label{R:septensor}
Let $X$ be a C*-correspondence over a C*-algebra $A$. A review of the arguments of the proof of Theorem \ref{T:hyp}, and using Theorem \ref{T:hypiffuep}, yield that the following are equivalent:
\begin{enumerate}
\item Every Cuntz-Pimsner representation $(H,\pi,t)$ is maximal on $\T_X^+$.
\item $\pi_X\times t_X(\T_X^+)\subseteq \O_X$ is hyperrigid and separating.
\item $\phi_X(J_X)X=X$.
\end{enumerate}
In particular, if any of the above is satisfied then $X$ is non-degenerate.
\end{remark}

\begin{remark}\label{R:bb}
In a recent work Bilich \cite{Bil24} recovers Kim's characterisation \cite[Corollary 4.9]{Bil24} by using a characterisation of maximality for Cuntz-Pimsner representations \cite[Theorem 4.4]{Bil24}. In \cite[Definition 2.6]{Bil24}, a non-degenerate C*-correspondence is said to be hyperrigid if all its Cuntz-Pimsner covariant representations are maximal. Since the tensor algebra of a non-degenerate C*-correspondence is separating, this definition aligns with our results. Specifically, Theorem \ref{T:hypiffuep} establishes that for separating generating sets, hyperrigidity is equivalent to maximality for all $*$-representations.
\end{remark}

\end{document}